\documentclass[11pt]{article}
\usepackage{amssymb,amsmath}% font used for R in Real numbers
\usepackage{mathrsfs}
\usepackage{latexsym}
\usepackage{color}
\usepackage{amssymb}
\usepackage{amsmath}
\usepackage{amsthm}
\allowdisplaybreaks

\catcode`!=11
\let\!int\int \def\int{\displaystyle\!int}
\let\!lim\lim \def\lim{\displaystyle\!lim}
\let\!sum\sum \def\sum{\displaystyle\!sum}
\let\!sup\sup \def\sup{\displaystyle\!sup}
\let\!inf\inf \def\inf{\displaystyle\!inf}
\let\!cap\cap \def\cap{\displaystyle\!cap}
\let\!max\max \def\max{\displaystyle\!max}
\let\!min\min \def\min{\displaystyle\!min}
\catcode`!=12

\newtheorem{theorem}{\bf Theorem}[section]
\newtheorem{lemma}{\bf Lemma}[section]
\newtheorem{definition}{\bf Definition}[section]
\newtheorem{proposition}{\bf Proposition}[section]
\newtheorem{corollary}{\bf Corollary}[section]
\newtheorem{remark}{\bf Remark}[section]

\catcode`@=11
\@addtoreset{equation}{section}
\catcode`@=12

\begin{document}

\title{Limiting dynamics for stochastic nonclassical diffusion equations}
\author{Peng Gao
\\[2mm]
\small School of Mathematics and Statistics, and Center for Mathematics
\\
\small and Interdisciplinary Sciences, Northeast Normal University,
\\
\small Changchun 130024,  P. R. China
\\[2mm]
\small Email: gaopengjilindaxue@126.com }
\date{}
\maketitle

\vbox to -13truemm{}

\begin{abstract}
\par
In this paper, we are concerned with the dynamical behavior of the stochastic nonclassical parabolic equation,
more precisely, it is shown that the inviscid limits of the stochastic nonclassical diffusion equations reduces to the stochastic heat equations.
We deal with initial values in $H^{1}_{0}(I)$ and $H^{2}(I)\cap H^{1}_{0}(I)$. When the initial value in $H^{1}_{0}(I),$
we establish the inviscid limits of the weak martingale solution;
when the initial value in $H^{2}(I)\cap H^{1}_{0}(I),$ we establish the inviscid limits of the weak solution, the convergence in probability in $L^{2}(0,T;H^{1}(I))$ is proved.
The results are valid for cubic nonlinearity.
\par The key points in the proof of our convergence results are establishing some uniform estimates and the regularity theory for the solutions of the stochastic nonclassical diffusion equations which are independent of the parameter. Based on the uniform estimates, the tightness of distributions of the solutions can be obtained.
\\[6pt]
\end{abstract}
\section{Introduction}
\par
Nonclassical parabolic equation
$$u_{t}-\Delta u_{t}-\Delta u+u^{3}-u=0$$
arises as a model to describe physical phenomena such
as non-Newtonian flow, soil mechanics and heat conduction, etc.; see \cite{A1,C1,P1,T1,T2} and
references therein. Aifantis \cite{A1} provides a quite general approach for obtaining these equations.
\par
In a number of applications, the systems are
subject to stochastic fluctuations arising as a result of
either uncertain forcing (stochastic external forcing) or uncertainty
of the governing laws of the system. The need for taking random effects into account in modeling, analyzing, simulating and predicting
complex phenomena has been widely recognized in geophysical and climate dynamics,
materials science, chemistry, biology and other areas. Stochastic partial differential
equations (SPDEs or stochastic PDEs) are appropriate mathematical models for complex systems
under random influences \cite{W3}. The fact that in physical experiments there are always small irregularities
which give birth to a new random phenomenon justifies the
study of equations with noise.
\par
In this paper, we investigate
\begin{equation}\label{13}
\begin{array}{l}
\left\{
\begin{array}{llll}
d(u^{\varepsilon}-\varepsilon u^{\varepsilon}_{xx})+(-u^{\varepsilon}_{xx}+u^{\varepsilon3}-u^{\varepsilon})dt=g(u^{\varepsilon})dB
\\u^{\varepsilon}(0,t)=0=u^{\varepsilon}(1,t)
\\u^{\varepsilon}(0)=u_{0}

\end{array}
\right.
\end{array}
\begin{array}{lll}
{\rm{in}}~I\times(0,T)\\
{\rm{in}}~(0,T)\\
{\rm{in}}~I,
\end{array}
\end{equation}
where $\varepsilon\in [0,1), I=[0,1], T>0.$ This paper is concerned with the asymptotic behavior of solutions of (\ref{13}) as $\varepsilon\rightarrow0.$
\par
For the deterministic nonclassical diffusion equation
$$u_{t}-\varepsilon\Delta u_{t}-\Delta u+u^{3}-u=0,$$
\cite{W1} establishs some
uniform decay estimates for the solutions which are independent of the parameter $\varepsilon$,
then they prove the continuity of solutions as $\varepsilon\rightarrow0.$ Upper semicontinuity of the family of global attractors at $\varepsilon=0$ in the topology of $H^{1}_{0}$ is also established.
\cite{A2} considers the first initial boundary value problem for the non-autonomous nonclassical
diffusion equation. By using the asymptotic a priori estimate method, the authors prove the existence of
pullback attractors and the upper semicontinuity of pullback attractors.
\par
For the stochastic nonclassical diffusion equations, \cite{Z1} concerns the dynamics of this
equation on $\mathbb{R}^{N}$ perturbed by a $\varepsilon$-random term. By using an energy approach, the authors prove the asymptotic compactness
of the associated random dynamical system, and then the existence of
random attractors. Finally, they show the upper semicontinuity of
random attractors in the sense of Hausdorff semi-metric.
\cite{B1,Z2} prove the existence of pullback attractor for stochastic nonclassical diffusion equations on unbounded domains with non-autonomous deterministic and stochastic forcing terms, and by using a tail-estimates method, the authors establish the pullback asymptotic compactness of the random dynamical system.
\par
In recent years, many efforts have been devoted to studying the singularly perturbed nonlinear
SPDEs.
\par
\cite{C2,C3,C4,C5,C6} consider the Smoluchowski-Kramers approximation the singularly perturbed nonlinear
stochastic wave equations. In \cite{L1} relations between the asymptotic behavior for a stochastic wave equation and a heat equation
are considered. The upper semicontinuity of global random attractor and the global attractor of the heat equation is investigated. Furthermore they shows that the stationary solutions of the stochastic wave equation converge in probability to some stationary solution of the heat equation.
\cite{W2} studies a continuity property for the measure attractors of the singularly perturbed nonlinear
stochastic wave equations, any one stationary solution of the limit heat equation
is a limit point of a stationary solution of the singularly perturbed nonlinear
stochastic wave equations. An averaging method is applied to derive effective approximation to a singularly
perturbed nonlinear stochastic damped wave equation in \cite{L2}. \cite{L3} establishes a large deviation principle for the singularly perturbed stochastic
nonlinear damped wave equations. In \cite{L4}, the random inertial manifold of a stochastic damped nonlinear wave equations with singular
perturbation is proved to be approximated almost surely by that of a stochastic
nonlinear heat equation which is driven by a new Wiener process depending on the
singular perturbation parameter.
\par
\cite{R2} establishs the weak martingale solution for stochastic model for two-dimensional second grade fluids and studied their behaviour when $\alpha\rightarrow0$. \cite{D3} studies the asymptotic behavior of weak solutions to the stochastic 3D Navier-Stokes-$\alpha$
model as $\alpha\rightarrow0$, the main result provides a new construction of the weak
solutions of stochastic 3D Navier-Stokes equations as approximations by sequences of
solutions of the stochastic 3D Navier-Stokes-$\alpha$ model. \cite{S4} discusses the relation of the stochastic 3D magnetohydrodynamic-$\alpha$
model to the stochastic 3D magnetohydrodynamic equations by proving a convergence theorem, that is,
as the length scale $\alpha\rightarrow0$, a subsequence of weak martingale solutions of the stochastic 3D
magnetohydrodynamic-$\alpha$ model converges to a certain weak martingale solution of the stochastic 3D magnetohydrodynamic
equations.
\par
However, there are very few results for the limiting dynamics for stochastic nonclassical diffusion equations with singularly perturbed.

\par
Motivated by previous research and from both physical and mathematical standpoints, the following mathematical questions arise naturally which are important from the point of view of dynamical systems:
\begin{itemize}
  \item Does the solution $u^{\varepsilon}$ for (\ref{13}) converge as $\varepsilon\rightarrow0$?
  \item If $u^{\varepsilon}$ converges as $\varepsilon\rightarrow0$, what is the limit of $u^{\varepsilon}$?
\end{itemize}
\par
In this paper we will answer the above problems. The question of asymptotic analysis of partial differential equations when some physical
parameters converge to some limit has always been of great interest.
\par To the best
of our knowledge, it is the first contribution to the literature on this problem.

\par
Through this paper, we make the following assumptions:
\par
H1)
Let $(\Omega,\mathcal {F},\{\mathcal {F}_{t}\}_{t\geq0},P)$ be a
complete filtered probability space on which a one-dimensional
standard Brownian motion $\{B(t)\}_{t\geq 0}$ is defined such that
$\{\mathcal {F}_{t}\}_{t\geq0}$ is the natural filtration generated
by $w(\cdot),$ augmented by all the $P-$ null sets in $\mathcal
{F}.$ Let $H$ be a Banach space, and let $C([0,T];H)$ be the Banach
space of all $H-$valued strongly continuous functions defined on
$[0,T].$ We denote by $L_{\mathcal {F}}^{p}(0,T;H)(1\leq p<+\infty)$ the Banach space
consisting of all $H-$valued $\{\mathcal {F}_{t}\}_{t\geq0}-$adapted
processes $X(\cdot)$ such that
$E(\|X(\cdot)\|^{p}_{L^{p}(0,T;H)})<\infty;$ by
$L_{\mathcal {F}}^{\infty}(0,T;H)$ the Banach space consisting of
all $H-$valued ${\mathcal \{\mathcal {F}_{t}}\}_{t\geq0}-$adapted bounded
processes; by $L_{\mathcal {F}}^{2}(\Omega;C([0,T];H))$ the
Banach space consisting of all $H-$valued $\{\mathcal
{F}_{t}\}_{t\geq0}-$adapted continuous processes $X(\cdot)$ such
that $E(\|X(\cdot)\|^{2}_{C([0,T];H)})<\infty.$ All the
above spaces are endowed with the canonical norm.
\par
H2) For a random variable $\xi$, we denote by $\mathcal{L}(\xi)$ its distribution.
\par
H3) $(\cdot,\cdot)$ stands for the inner product in $L^{2}(I)$.
\par
H4) The letter $C$ with or without subscripts denotes positive constants
whose value may change in different occasions. We will write the dependence
of constant on parameters explicitly if it is essential.
\par
We make the the two different assumptions on $g.$
\par
(A) $g\in C(\mathbb{R})$ and there exists a constant $L>0$ such that
\begin{equation*}
\begin{array}{l}
\begin{array}{llll}
\|g(u)\|_{L^{2}(I)}\leq L(1+\|u\|_{L^{2}(I)})~~\forall u\in L^{2}(I),
\\\|g(u_{1})-g (u_{2})\|_{L^{2}(I)}\leq L\|u_{1}-u_{2}\|_{L^{2}(I)}~~\forall u_{1},u_{2} \in L^{2}(I).
\end{array}
\end{array}
\end{equation*}
\par
(B) $g\in C(\mathbb{R})$ and there exists a constant $L>0$ such that
\begin{equation*}
\begin{array}{l}
\begin{array}{llll}
\|g(u)\|_{L^{2}(I)}\leq L(1+\|u\|_{L^{2}(I)})~~\forall u \in L^{2}(I),
\\\|g(u)\|_{H^{1}(I)}\leq L(1+\|u\|_{H^{1}(I)})~~\forall u \in H^{1}(I),
\\\|g(u_{1})-g(u_{2})\|_{H^{1}(I)}\leq L\|u_{1}-u_{2}\|_{H^{1}(I)}~~\forall u_{1},u_{2} \in H^{1}(I)

.
\end{array}
\end{array}
\end{equation*}
\subsection{Weak martingale solution}
\begin{definition}
A weak martingale solution of (\ref{13})
is a system $\{(\Omega,\mathcal{F},\mathbb{P}),$ $(\mathcal{F}_{t})_{0\leq t\leq T},u,B\},$ where
\par
(1) $(\Omega,\mathcal{F},\mathbb{P})$ is a complete probability space,
\par
(2) $(\mathcal{F}_{t})_{0\leq t\leq T}$ is a filtration satisfying the usual condition on $(\Omega,\mathcal{F},\mathbb{P}),$
\par
(3) $B$ is a $\mathcal{F}_{t}-$adapted $\mathbb{R}-$valued Wiener process,
\par
(4) $u\in L^{p}(\Omega,L^{\infty}(0,T;L^{2}(I)))\cap L^{p}(\Omega,L^{2}(0,T;H^{1}(I)))\cap L^{2p}(\Omega,L^{4}(0,T;L^{4}(I))),$ for every $1\leq p\leq\infty,$
\par
(5) For all $\varphi\in H^{1}_{0}(I),$
\begin{equation*}
\begin{array}{l}
\begin{array}{llll}
[(u(t),\varphi)+\varepsilon(u_{x}(t),\varphi_{x})]-[(u_{0},\varphi)+\varepsilon(u_{0x},\varphi_{x})]+\int_{0}^{t}((u_{x},\varphi_{x})+(u^{3}-u,\varphi))ds
\\~~~~~~~~~~~~~~~~~=\int_{0}^{t}(g(u),\varphi)dB
\end{array}
\end{array}
\end{equation*}
hold $dt\otimes d\mathbb{P}-$almost everywhere.
\par
(6) The function $u(t)$ take values in $L^{2}(I)$ and is continuous with respect to $t$ $\mathbb{P}-$almost surely.
\end{definition}

\par
The first main result of this paper is given in the next statement.
\begin{theorem}\label{Th3}
Let assumption (A) be satisfied, $ T>0$ and $u_{0}\in H^{1}_{0}(I).$ For any $\varepsilon\in [0,\frac{1}{2}],$ there exists a weak martingale solution $\{(\Omega^{\varepsilon},\mathcal{F}^{\varepsilon},\mathbb{P}^{\varepsilon}),(\mathcal{F}^{\varepsilon}_{t})_{0\leq t\leq T},u^{\varepsilon},B^{\varepsilon}\}$ of problem (\ref{13}) such that the following estimates hold for any $1\leq p< \infty:$
\begin{eqnarray}
\label{33}\displaystyle \mathbb{E} \sup\limits_{0\leq t\leq T}(\|u^{\varepsilon}(t)\|_{L^{2}(I)}^{2}+\varepsilon\|u_{x}^{\varepsilon}(t)\|_{L^{2}(I)}^{2})^{\frac{p}{2}}\leq C(p,T),
\\\label{34}\displaystyle \mathbb{E}\left(\int_{0}^{T}(\|u_{x}^{\varepsilon}(t)\|_{L^{2}(I)}^{2}+\|u^{\varepsilon}\|_{L^{4}(I)}^{4})dt\right)^{\frac{p}{2}}\leq C(p,T),
\\\label{35}\displaystyle \mathbb{E} \sup\limits_{0\leq |\theta|\leq \delta\leq 1} \int_{0}^{T}\|u^{\varepsilon}(t+\theta)-u^{\varepsilon}(t)\|_{H^{-1}(I)}^{2}dt\leq C(p,T)\delta,
\end{eqnarray}
where $C(p,T)$ is a constant independent of $\varepsilon.$
\par
Moreover, let $u_{1}$ and $u_{2}$ be two weak martingale solutions of problem (\ref{13}) defined on the same prescribed
stochastic basis $\{(\Omega,\mathcal{F},P),(\mathcal{F}_{t})_{0\leq t\leq T},B\}$ starting with the same initial condition $u_{0},$ then
$$u_{1}=u_{2}~~~~~P-{\rm{a.s.~~~ for~~ all~~}} t\in [0,T].$$
\end{theorem}
\begin{remark}
If we replace $g(u)$ in (\ref{13}) by $g(t,u)$ and assume that $g(t,u)$ is nonlinear measurable mapping defined on $[0,T] \times L^{2}(I)$ taking values on $L^{2}(I),$ it is continuous with respect to $u$
and there exists  a constant $C$ such that
\begin{equation*}
\begin{array}{l}
\begin{array}{llll}
\|g(t,u)\|_{L^{2}(I)}\leq C(1+\|u\|_{L^{2}(I)})~~\forall t\in[0,T]~~\forall u\in L^{2}(I),
\\\|g(t,u_{1})-g (t,u_{2})\|_{L^{2}(I)}\leq C\|u_{1}-u_{2}\|_{L^{2}(I)}~~\forall u_{1},u_{2} \in L^{2}(I),
\end{array}
\end{array}
\end{equation*}
the conclusion in Theorem \ref{Th3} also holds.
\end{remark}
\begin{remark}
Theorem \ref{Th3} is established by the compactness method combines the Galerkin approximation
scheme with sharp compactness results in function spaces of
Sobolev type due to Simon and some celebrated probabilistic
compactness results of Prokhorov and Skorokhod.
\end{remark}
\par
Asymptotic behavior of the weak martingale solutions for the stochastic nonclassical diffusion equations as $\varepsilon\rightarrow0$ can be described by the following results.
\begin{theorem}\label{Th4}
Let assumption (A) be satisfied, $ T>0$ and $u_{0}\in H^{1}_{0}(I).$ If $\{(\Omega^{\varepsilon},\mathcal{F}^{\varepsilon},\mathbb{P}^{\varepsilon}),(\mathcal{F}^{\varepsilon}_{t})_{0\leq t\leq T},u^{\varepsilon},B^{\varepsilon}\}_{\varepsilon\in [0,1]}$ are the weak martingale solutions of problem (\ref{13}), there exists a subsequence $\{\varepsilon_{i}\}\subset[0,1]$ with $\varepsilon_{i}\rightarrow0$ as $i\rightarrow\infty $, a probability space $(\Omega,\mathcal{F},\mathbb{P})$ and random variables $(\tilde{u}^{\varepsilon_{i}},\tilde{B}^{\varepsilon_{i}}),$  $(u,B)$ on $(\Omega,\mathcal{F},\mathbb{P})$ with values in $  L^{2}(0,T;L^{2}(I))\times C([0,T];\mathbb{R}^{1})$ such that
\begin{eqnarray*}
\begin{array}{l}
\begin{array}{llll}
\mathcal{L}(\tilde{u}^{\varepsilon_{i}},\tilde{B}^{\varepsilon_{i}})=\mathcal{L}(u^{\varepsilon_{i}},B^{\varepsilon_{i}})
\end{array}
\end{array}
\end{eqnarray*}
and the following convergences hold for any $1\leq p< \infty:$
\begin{eqnarray*}
\begin{array}{l}
\begin{array}{llll}
\tilde{u}^{\varepsilon_{i}}\rightarrow u~~{\rm{strongly~~ in}}~~ L^{2}(\Omega,L^{2}(0,T;L^{2}(I))),
\\
\tilde{u}^{\varepsilon_{i}}\rightarrow u~~{\rm{weakly~~ in}}~~ L^{p}(\Omega,L^{2}(0,T;H^{1}(I))),
\\
\tilde{u}^{\varepsilon_{i}}\rightarrow u~~{\rm{weakly~~star~~ in}}~~ L^{p}(\Omega,L^{\infty}(0,T;L^{2}(I))),
\\
\tilde{B}^{\varepsilon_{i}}\rightarrow B~~{\rm{in}}~~C([0,T];\mathbb{R}^{1})~~\mathbb{P}-a.s.,

\end{array}
\end{array}
\end{eqnarray*}
as $i\rightarrow\infty $ and $\{(\Omega,\mathcal{F},\mathbb{P}),(\mathcal{F}_{t})_{0\leq t\leq T},u,B\}$ is a weak martingale solution of problem
\begin{equation}\label{61}
\begin{array}{l}
\left\{
\begin{array}{llll}
du+(-u_{xx}+u^{3}-u)dt=g(u)dB
\\u(0,t)=0=u(1,t)
\\u(0)=u_{0}

\end{array}
\right.
\end{array}
\begin{array}{lll}
{\rm{in}}~I\times(0,T)\\
{\rm{in}}~(0,T)\\
{\rm{in}}~I.
\end{array}
\end{equation}
\end{theorem}
\begin{remark}
If we replace $g(u)$ in (\ref{13}) by $g(t,u)$ and assume that $g(t,u)$ is nonlinear measurable mapping defined on $[0,T] \times L^{2}(I)$ taking values on $L^{2}(I),$ it is continuous with respect to $u$
and there exists  a constant $C$ such that
\begin{equation*}
\begin{array}{l}
\begin{array}{llll}
\|g(t,u)\|_{L^{2}(I)}\leq C(1+\|u\|_{L^{2}(I)})~~\forall t\in[0,T]~~\forall u\in L^{2}(I),
\\\|g(t,u_{1})-g (t,u_{2})\|_{L^{2}(I)}\leq C\|u_{1}-u_{2}\|_{L^{2}(I)}~~\forall u_{1},u_{2} \in L^{2}(I),
\end{array}
\end{array}
\end{equation*}
the conclusion in Theorem \ref{Th4} also holds.
\end{remark}
\subsection{Weak solution}
\par
Next, we consider another kind of solution to (\ref{13}).

\begin{definition}
A stochastic process $u$ is said to be a weak solution of (\ref{13})
if
\par
$u$ is $L^{2}(I)$-valued and $\mathcal{F}_{t}$-measurable for each $t\in [0,T],$
\par
$u\in L^{2}(\Omega;L^{2}(C([0,T];L^{2}(I))),$
\par
$u(0)=u_{0}$
\par
and
\begin{equation}
\begin{array}{l}
\begin{array}{llll}
(u(t),\varphi)-\varepsilon(u(t),\varphi_{xx})
\\~~~~~~~~~~~=(u_{0},\varphi)-\varepsilon(u_{0},\varphi_{xx})+\int_{0}^{t}( u(s),\varphi_{xx})ds-\int_{0}^{t}(u^{3}-u,\varphi)ds+\int_{0}^{t}(g(s),\varphi)dB(s)
\end{array}
\end{array}
\end{equation}
holds for all $t\in [0,T]$ and all $\varphi\in H^{2}(I)\cap  H^{1}_{0}(I)$, for almost all $\omega\in \Omega.$
\end{definition}
\begin{remark} The weak solution of SPDEs has been discussed in \cite{D2}.
\end{remark}
\begin{theorem}\label{Th1}
Let assumption (B) be satisfied, $T>0$ and $u_{0}\in H^{2}(I)\cap H^{1}_{0}(I).$ For any $\varepsilon\in [0,\frac{1}{2}], $ there exists a unique weak solution $u^{\varepsilon}(t)$ to (\ref{13}) in $ L^{2}(\Omega;C([0,T];H^{2}(I)\cap H^{1}_{0}(I)))$ and
for any $1\leq p< \infty,$ there exists a constant $C(p,L,T,I,u_{0})$ such that
\begin{equation}\label{15}
\begin{array}{l}
\begin{array}{llll}
\mathbb{E} \sup\limits_{0\leq t\leq T}\|u^{\varepsilon}(t)\|_{L^{2}(I)}^{2p}+\mathbb{E}(\int_{0}^{T}\| u^{\varepsilon}_{x}\|_{L^{2}(I)}^{2}dt)^{p}+\mathbb{E}(\int_{0}^{T}\int_{I}u^{\varepsilon4}dxdt)^{p}+\mathbb{E}(\int_{0}^{T}\varepsilon\| u^{\varepsilon}_{xx}\|_{L^{2}(I)}^{2}dt)^{p}
\\\leq C(p,L,T,I,u_{0}).
\end{array}
\end{array}
\end{equation}
Moreover, there exists a constant $C(L,T,I,u_{0})$ such that
\begin{equation}\label{16}
\begin{array}{l}
\begin{array}{llll}
\mathbb{E} \sup\limits_{0\leq t\leq T}(\|u^{\varepsilon}_{x}(t)\|_{L^{2}(I)}^{2}+\varepsilon\|u^{\varepsilon}_{xx}(t)\|_{L^{2}(I)}^{2})+\mathbb{E}\int_{0}^{T}\| u^{\varepsilon}_{xx}\|_{L^{2}(I)}^{2}dt\leq C(L,T,I,u_{0}).

\end{array}
\end{array}
\end{equation}

\end{theorem}
\begin{remark}
Since nonlinear terms $u^{3}-u$ are not Lipschitz continuous,
we will use a truncation argument which will lead
to a local existence result. Then via some a priori estimates we
obtain that the solution is also global.
\end{remark}
\par
Asymptotic behavior of the weak solutions for the stochastic nonclassical diffusion equations as $\varepsilon\rightarrow0$ can be described by the following results.
\begin{theorem}\label{Th2}
Let assumption (B) be satisfied, $T>0$ and $u_{0}\in H^{2}(I)\cap H^{1}_{0}(I).$ For any $\varepsilon\in [0,\frac{1}{2}],$ if $u^{\varepsilon}$ is the weak solution to (\ref{13}) and $z$ is the weak solution to
\begin{equation}\label{21}
\begin{array}{l}
\left\{
\begin{array}{llll}
dz+(-z_{xx}+z^{3}-z)dt=g(z)dB
\\z(0,t)=0=z(1,t)
\\z(0)=u_{0}

\end{array}
\right.
\end{array}
\begin{array}{lll}
{\rm{in}}~I\times(0,T)\\
{\rm{in}}~(0,T)\\
{\rm{in}}~I,
\end{array}
\end{equation}
then $u^{\varepsilon}$ converges in probability to $z$ in $L^{2}(0,T;H^{1}(I))$ as $\varepsilon\rightarrow0,$
namely, for any $\delta>0,$ we have
\begin{equation}
\begin{array}{l}
\begin{array}{llll}
\lim\limits_{\varepsilon\rightarrow0}\mathbb{P}( \|u^{\varepsilon}-z\|_{L^{2}(0,T;H^{1}(I))}>\delta)=0.

\end{array}
\end{array}
\end{equation}

\end{theorem}
\subsection{Main difficulties}
\par
The main difficulties in this paper are the following respects:
\par
\begin{itemize}
  \item  Multiplicative type noise. The noise in equation (\ref{13}) is not additive type, (\ref{13}) is perturbed by a stochastic term of multiplicative type, thus the method in \cite{W1,Z1,Z2} can not be used in dealing with (\ref{13}), we should take new measure. Here the
presence of a diffusion coefficient $g$ in front of the stochastic perturbation which is nonconstant
makes the proof of Theorem \ref{Th4} and Theorem \ref{Th2} definitely
more delicate and requires some extra work which is not necessary in the case of a Gaussian
perturbation.
  \item ``BBM" term. Equation (\ref{13}) contains the ``BBM" term $-u_{xxt}$, its stochastic from is $-d u_{xx}$, this brings us new difficulty in establishing the existence and regularity theory for the stochastic nonclassical diffusion equations. In the present work we will
try to overcome this difficulty by developing the Galerkin approximation techniques in \cite{K1,G2,G3,G4}. The ``BBM" term is different from the usual reaction-diffusion equation
essentially. For example, the nonclassical diffusion equation does not have smoothing effect, e.g., if the
initial data only belongs to a weaker topology space, the solution can not belong to a stronger topology
space with higher regularity. Moreover, since the existence of this term, we can't use the It\"{o} formula to $u^{2}.$
We borrow an essential idea from \cite{K1,G2,G3,G4}, but substantial technical adaptation is necessary for the problem in this paper.
\item Uniform estimates independent of the parameter $\varepsilon$. Since the parameter $\varepsilon$ in singular perturbation problem (1.1) is small, the uniform estimates for the solution of (\ref{13}) which are independent of the parameter $\varepsilon$ are very hard to obtain. The proof of the convergence result requires
uniform estimates on the Sobolev regularity in space and in time for the solutions to the stochastic nonclassical diffusion equation. As known, such
uniform bounds are used to establish tightness property of $u^{\varepsilon}$ in an appropriate functional space.

\item The cubic non-linear term. The last difficulty arises from polynomial nonlinearity in equation (\ref{13}), the nonlinear term in (\ref{13}) is cubic term $u^{3}-u$, the main obstacle is that it is difficult to obtain a higher regularity estimate to
guarantee the continuous convergence of the solutions as $\varepsilon\rightarrow0$. This type of nonlinearity can be handled by the truncation method. In order to overcome the problem, we use the cut-off technique and the Gagliardo-Nirenberg inequality.

\end{itemize}

\par
This paper is organized as follows. In Section 2, we give some preliminaries and gather all the necessary tools. The existence of weak martingale solutions
for (\ref{13}) is discussed in Section 3, we introduce a Galerkin
approximation scheme for the problem (\ref{13}) and obtain a priori estimates for the approximating
solutions, then we prove the crucial result of tightness of Galerkin¡¯s solutions and apply
Prokhorov¡¯s and Skorokhod¡¯s compactness results to prove Theorem \ref{Th3}. Section 4 is concerned with the continuity of weak martingale solutions
for (\ref{13}) as $\varepsilon\rightarrow0$. We derive the results of the tightness of the corresponding probability measures
and perform the passage to the limit which establishes the convergence of weak martingale solutions. In Section 5, applying the Picard iteration method to the
corresponding truncated equation, we give the local existence of weak solutions to (\ref{13}). Then, the energy estimate shows that the weak solution is also global in time. Moreover, we obtain the uniform estimates for the solution of (\ref{13}) which are independent of the parameter $\varepsilon$.
Section 6 is concerned with the continuity of weak solutions for (\ref{13}) as $\varepsilon\rightarrow0$. We derive tightness property of weak solutions in $L^{2}(0,T;H^{1}(I))$ and perform the passage to the limit which establishes the convergence of weak solutions.

\section{Preliminary}
This section is devoted to some preliminaries for the proof of Theorem \ref{Th3}--Theorem \ref{Th2}.
\subsection{Some tools}
\par
The following compactness results is important for tightness property of Galerkin solutions.
\begin{lemma}(See \cite[Theorem 5]{S1})\label{L7}
Let $X,B$ and $Y$ be some Banach spaces such that $X$ is compactly embedded into $B$ and let $B$ be a subset of $Y.$ For
any $1\leq p,q \leq\infty,$ let $V$ be a set bounded in $L^{q}(0,T;X)$ such that
$$\lim_{\theta\rightarrow 0}\int_{0}^{T-\theta}\|v(t+\theta)-v(t)\|_{Y}^{p}dt=0,$$
uniformly for all $v\in V.$ Then $V$ is relatively compact in $L^{p}(0,T;B).$
\end{lemma}
\par
According to Lemma \ref{L7}, we can obtain the following compactness result.
\begin{corollary}\label{c1}
Let $X,B$ and $Y$ satisfy the same assumptions in Lemma \ref{L7} and $\mu_{m},\nu_{m}$ be two sequences which converge to zero as $m\rightarrow\infty.$
Then \begin{eqnarray*}
\begin{array}{l}
\mathcal{Z}=
\left\{q\in\left|
\begin{array}{llll}
L^{2}(0,T; X)\cap L^{\infty}(0,T;B)
\\ \displaystyle\sup\limits_{m}\frac{1}{\nu_{m}}\sup\limits_{|\theta|\leq \mu_{m}}\left(\int_{0}^{T}\|q(t+\theta)-q(t)\|_{Y}^{2}dt\right)^{\frac{1}{2}}<+\infty
\end{array}
\right.\right\}
\end{array}
\end{eqnarray*}
in $L^{2}(0,T;B)$ is compact.
\end{corollary}
\begin{remark}
The above compactness result plays a crucial role in the proof of the tightness of the probability measures generated
by the sequence $\{u^{\varepsilon}\}_{\varepsilon>0}.$
\end{remark}

\par
Now we introduce several spaces which will be used in the next section. Let $\mu_{m},\nu_{m}$ be two sequences that defined in Corollary \ref{c1}.
\par
$\bullet$ The space $Y_{\mu_{m},\nu_{m}}^{1}$ is a Banach space with the norm
\begin{eqnarray*}
\begin{array}{l}
\begin{array}{llll}
\displaystyle\|y\|_{Y_{\mu_{m},\nu_{m}}^{1}}=\sup\limits_{0\leq t\leq T}\|y(t)\|_{L^{2}(I)}+\left(\int_{0}^{T}\|y(t)\|_{H^{1}(I)}^{2}dt\right)^{\frac{1}{2}}
\\~~~~~~~~~~~~~~~~\displaystyle+\sup\limits_{m}\frac{1}{\nu_{m}}\sup\limits_{|\theta|\leq \mu_{m}}\int_{0}^{T-\theta}\|y(t+\theta)-y(t)\|_{H^{-1}(I)}^{2}dt.
\end{array}
\end{array}
\end{eqnarray*}
\par
$X_{p,\mu_{m},\nu_{m}}^{1}$ is a space consist of all random variables $y$ on $(\Omega,\mathcal{F},\mathbb{P})$ which satisfy
\begin{eqnarray*}
\begin{array}{l}
\begin{array}{llll}
\displaystyle \mathbb{E}\sup\limits_{0\leq t\leq T}\|y(t)\|_{L^{2}(I)}^{2p}<\infty,~~~\mathbb{E}\left(\int_{0}^{T}\|y(t)\|_{H^{1}(I)}^{2}dt\right)^{\frac{p}{2}}<\infty,
\\ \displaystyle \mathbb{E}\sup\limits_{m}\frac{1}{\nu_{m}}\left(\sup\limits_{|\theta|\leq \mu_{m}}\int_{0}^{T-\theta}\|y(t+\theta)-y(t)\|_{H^{-1}(I)}^{2}dt\right)^{\frac{1}{2}}<\infty,

\end{array}
\end{array}
\end{eqnarray*}
where $\mathbb{E}$ denotes the mathematical expectation with respect to the probability measure $\mathbb{P}$.
Endowed with the norm
\begin{eqnarray*}
\begin{array}{l}
\begin{array}{llll}
\displaystyle\|y\|_{X_{p,\mu_{m},\nu_{m}}^{1}}=\left(\mathbb{E}\sup\limits_{0\leq t\leq T}\|y(t)\|_{L^{2}(I)}^{2p}\right)^{\frac{1}{2p}}+\left(\mathbb{E}(\int_{0}^{T}\|y(t)\|_{H^{1}(I)}^{2}dt)^{\frac{p}{2}}\right)^{\frac{2}{p}}
\\\displaystyle~~~~~~~~~~~~~~~~~~~~~+\mathbb{E}\sup\limits_{m}\frac{1}{\nu_{m}}\left(\sup\limits_{|\theta|\leq \mu_{m}}\int_{0}^{T-\theta}\|y(t+\theta)-y(t)\|_{H^{-1}(I)}^{2}dt\right)^{\frac{1}{2}},
\end{array}
\end{array}
\end{eqnarray*}
$X_{p,\mu_{m},\nu_{m}}^{1}$ is a Banach space.
\par
$\bullet$ The space $Y_{\mu_{m},\nu_{m}}^{2}$ is a Banach space with the norm
\begin{eqnarray*}
\begin{array}{l}
\begin{array}{llll}
\displaystyle\|y\|_{Y_{\mu_{m},\nu_{m}}^{2}}=\sup\limits_{0\leq t\leq T}\|y(t)\|_{H^{1}(I)}+\left(\int_{0}^{T}\|y(t)\|_{H^{2}(I)}^{2}dt\right)^{\frac{1}{2}}
\\~~~~~~~~~~~~~~~~\displaystyle+\sup\limits_{m}\frac{1}{\nu_{m}}\sup\limits_{|\theta|\leq \mu_{m}}\int_{0}^{T-\theta}\|y(t+\theta)-y(t)\|_{L^{2}(I)}^{2}dt.
\end{array}
\end{array}
\end{eqnarray*}
\par
$X_{p,\mu_{m},\nu_{m}}^{2}$ is a space consist of all random variables $y$ on $(\Omega,\mathcal{F},\mathbb{P})$ which satisfy
\begin{eqnarray*}
\begin{array}{l}
\begin{array}{llll}
\displaystyle \mathbb{E}\sup\limits_{0\leq t\leq T}\|y(t)\|_{H^{1}(I)}^{2p}<\infty,~~~\mathbb{E}\left(\int_{0}^{T}\|y(t)\|_{H^{2}(I)}^{2}dt\right)^{\frac{p}{2}}<\infty,
\\ \displaystyle \mathbb{E}\sup\limits_{m}\frac{1}{\nu_{m}}\left(\sup\limits_{|\theta|\leq \mu_{m}}\int_{0}^{T-\theta}\|y(t+\theta)-y(t)\|_{L^{2}(I)}^{2}dt\right)^{\frac{1}{2}}<\infty,

\end{array}
\end{array}
\end{eqnarray*}
where $\mathbb{E}$ denotes the mathematical expectation with respect to the probability measure $\mathbb{P}$.
Endowed with the norm
\begin{eqnarray*}
\begin{array}{l}
\begin{array}{llll}
\displaystyle\|y\|_{X_{p,\mu_{m},\nu_{m}}^{2}}=\left(\mathbb{E}\sup\limits_{0\leq t\leq T}\|y(t)\|_{H^{1}(I)}^{2p}\right)^{\frac{1}{2p}}+\left(\mathbb{E}(\int_{0}^{T}\|y(t)\|_{H^{2}(I)}^{2}dt)^{\frac{p}{2}}\right)^{\frac{2}{p}}
\\\displaystyle~~~~~~~~~~~~~~~~~~~~~+\mathbb{E}\sup\limits_{m}\frac{1}{\nu_{m}}\left(\sup\limits_{|\theta|\leq \mu_{m}}\int_{0}^{T-\theta}\|y(t+\theta)-y(t)\|_{L^{2}(I)}^{2}dt\right)^{\frac{1}{2}},
\end{array}
\end{array}
\end{eqnarray*}
$X_{p,\mu_{m},\nu_{m}}^{2}$ is a Banach space.

\par
In order to pass from martingale to pathwise solutions we make essential use of an elementary but powerful characterization of
convergence in probability as given in \cite{G1}.
\begin{lemma}(Gy\"{o}ngy-Krylov Theorem)(See \cite[Lemma 1.1]{G1},\cite[Proposition 6.3]{P3})\label{L1}
Let $E$ be a Polish space equipped with the Borel $\sigma$-algebra. A sequence
of $E$-valued random element $z_{n}$ converges in probability if and only if for every pair
of subsequences $z_{l},z_{m}$ there exists a subsequence $w_{k}=(z_{l(k)},z_{m(k)})$ converging weakly
to a random element $w$ supported on the diagonal $\{(x,y)\in E\times E:x=y\}.$
\end{lemma}

\par
Prokhorov's Theorem and Skorohod's Theorem will be used to establish the tightness of $u^{\varepsilon}.$
The following two lemmas will play crucial roles in the proof of Theorem \ref{Th1}.
\begin{lemma}[Prokhorov's Theorem]\label{L4}
A sequence of measures $\{\mu_{n}\}$ on $(E,\mathcal{B}(E))$ is tight if and only if it is relatively compact, that is there exists a
subsequence $\{\mu_{n_{k}}\}$ which weakly converges to a probability measure $\mu.$
\end{lemma}
\begin{lemma}[Skorohod's Theorem]\label{L5}
For an arbitrary sequence of probability measures $\{\mu_{n}\}$ on $(E,\mathcal{B}(E))$
weakly converges to a probability measure $\mu,$ there exists a probability space $(\Omega,\mathcal{F},P)$ and random variables $\xi,\xi_{1},...,\xi_{n},...$
with values in $E$ such that the probability law of $\xi_{n},$
$$\mathcal{L}(\mathcal{A})=P\{\omega \in \Omega:\xi_{n}(\omega)\in \mathcal{A}\},$$
for all $\mathcal{A}\in \mathcal{F},$ is $\mu_{n},$ the probability law of $\xi$ is $\mu,$ and $\lim\limits_{n\rightarrow \infty} \xi_{n}=\xi,$ $P-a.s.$

\end{lemma}

\subsection{The linear stochastic nonclassical diffusion equations}
This section is devoted to some preliminaries for the proof of Theorem \ref{Th1}.
\par
In this subsection, we let $G$ be the bounded domain of $\mathbb{R}^{n}(n\geq 1).$ We will use the results in this subsection with $n=1$ in Section 5.

\begin{definition}
A stochastic process $u$ is said to be a solution of
\begin{equation}\label{1}
\begin{array}{l}
\left\{
\begin{array}{llll}
d(u-\varepsilon\Delta u)+(-\Delta u+f)dt=gdB
\\u(x,t)=0
\\u(0)=u_{0}

\end{array}
\right.
\end{array}
\begin{array}{lll}
{\rm{in}}~G\times(0,T)\\
{\rm{in}}~\partial G\times(0,T)\\
{\rm{in}}~G,
\end{array}
\end{equation}
if
\par
$u$ is $L^{2}(G)$-valued and $\mathcal{F}_{t}$-measurable for each $t\in [0,T],$
\par
$u\in L^{2}(\Omega;C([0,T];L^{2}(G))),$
\par
$u(0)=u_{0}$
\par
and
\begin{equation}\label{4}
\begin{array}{l}
\begin{array}{llll}
(u(t),\varphi)-\varepsilon(u(t),\triangle \varphi)
\\~~~~~~~~~~~=(u_{0},\varphi)-\varepsilon(u_{0},\triangle\varphi)+\int_{0}^{t}( u(s),\triangle\varphi)ds-\int_{0}^{t}(f(s),\varphi)ds+\int_{0}^{t}(g(s),\varphi)dB(s)
\end{array}
\end{array}
\end{equation}
holds for all $t\in [0,T]$ and all $\varphi\in H^{2}(G)\cap  H^{1}_{0}(G)$, for almost all $\omega\in \Omega.$
\end{definition}
\begin{lemma}
(See \cite[Theorem 8.94]{R1})\label{L6}
There exists a set of positive real numbers $\{\lambda_{k}\}_{k\in N}$ such that the corresponding solutions $\{e_{k}\}_{k\in N}$ of the problem
\begin{eqnarray}
\begin{array}{l}
\left\{
\begin{array}{llll}
-\triangle e_{k}=\lambda_{k}e_{k}
\\e_{k}(x)=0
\end{array}
\right.
\end{array}
\begin{array}{lll}
{\rm{in}}~G
\\{\rm{on}}~\partial G
\end{array}
\end{eqnarray}
form a basis in $H^{2}(G)\cap H^{1}_{0}(G),$ which is orthonormal in $L^{2}(G).$
\end{lemma}
\begin{proposition}\label{Pro1}
For any $\varepsilon\in [0,1],$ there exists a constant $C$ independent of $\varepsilon.$
\par
1) If $u_{0}\in L^{2}(\Omega;L^{2}(G)), f\in L^{2}(\Omega;L^{2}(0,T;H^{-1}(G))), g\in L^{2}(\Omega;L^{2}(0,T;L^{2}(G))) $, then (\ref{1}) has a unique solution $u\in L^{2}(\Omega;C([0,T];L^{2}(G)))$ and
\begin{equation}\label{24}
\begin{array}{l}
\begin{array}{llll}
\mathbb{E} \sup\limits_{0\leq t\leq T}\|u(t)\|_{L^{2}(G)}^{2}
\leq C[\mathbb{E}\|u_{0}\|_{L^{2}(G)}^{2}+\mathbb{E} \int_{0}^{T}\|f(t)\|_{H^{-1}(G)}^{2}dt+\mathbb{E} \int_{0}^{T}\|g(t)\|_{L^{2}(G)}^{2}dt].
\end{array}
\end{array}
\end{equation}
\par
2) If $u_{0}\in L^{2}(\Omega;H^{1}_{0}(G)), f\in L^{2}(\Omega;L^{2}(0,T;H^{-1}(G))), g\in L^{2}(\Omega;L^{2}(0,T;L^{2}(G))) $, then (\ref{1}) has a unique solution $u\in L^{2}(\Omega;C([0,T];H^{1}_{0}(G)))\cap L^{2}(\Omega,L^{2}(0,T;H^{1}(G)))$ and
\begin{equation}\label{29}
\begin{array}{l}
\begin{array}{llll}
\mathbb{E} \sup\limits_{0\leq t\leq T}(\|u(t)\|_{L^{2}(G)}^{2}+\varepsilon\|\nabla u(t)\|_{L^{2}(G)}^{2})+\mathbb{E} \int_{0}^{T}\|\nabla u(t)\|_{L^{2}(G)}^{2}dt
\\\leq C[\mathbb{E}(\|u_{0}\|_{L^{2}(G)}^{2}+\|\nabla u_{0}\|_{L^{2}(G)}^{2})+\mathbb{E} \int_{0}^{T}\|f(t)\|_{H^{-1}(G)}^{2}dt+\mathbb{E} \int_{0}^{T}\|g(t)\|_{L^{2}(G)}^{2}dt].
\end{array}
\end{array}
\end{equation}
Moreover, it holds that
\begin{equation}
\begin{array}{l}
\begin{array}{llll}
(u(t),\varphi)+\varepsilon(\nabla u(t),\nabla\varphi)
\\~~~~~~~~~~~=(u_{0},\varphi)+\varepsilon(\nabla u_{0},\nabla\varphi)-\int_{0}^{t}(\nabla u(s),\nabla\varphi)ds-\int_{0}^{t}(f(s),\varphi)ds+\int_{0}^{t}(g(s),\varphi)dB(s)
\end{array}
\end{array}
\end{equation}
for all $t\in [0,T]$ and all $\varphi\in H^{1}_{0}(G)$, for almost all $\omega\in \Omega.$
\par
3) If $u_{0}\in L^{2}(\Omega;H^{2}(G)\cap H^{1}_{0}(G)), f\in L^{2}(\Omega;L^{2}(0,T;L^{2}(G))), g\in L^{2}(\Omega;L^{2}(0,T;H^{1}(G))) $, then (\ref{1}) has a unique solution $u\in L^{2}(\Omega;C([0,T];H^{2}(G)\cap H^{1}_{0}(G)))\cap L^{2}(\Omega,L^{2}(0,T;H^{2}(G)))$
and
\begin{equation}\label{8}
\begin{array}{l}
\begin{array}{llll}
\mathbb{E} \sup\limits_{0\leq t\leq T}(\|\nabla u(t)\|_{L^{2}(G)}^{2}+\varepsilon\|\triangle u(t)\|_{L^{2}(G)}^{2})+\mathbb{E} \int_{0}^{T}\|\triangle u(t)\|_{L^{2}(G)}^{2}dt
\\\leq C[\mathbb{E}(\|\nabla u_{0}\|_{L^{2}(G)}^{2}+\|\triangle u_{0}\|_{L^{2}(G)}^{2})+\mathbb{E} \int_{0}^{T}\|f(t)\|_{L^{2}(G)}^{2}dt+\mathbb{E} \int_{0}^{T}\|g(t)\|_{H^{1}(G)}^{2}dt].
\end{array}
\end{array}
\end{equation}
Moreover, it holds that
\begin{equation}
\begin{array}{l}
\begin{array}{llll}
(u(t),\varphi)-\varepsilon(\triangle u(t),\varphi)
\\~~~~~~~~~~~=(u_{0},\varphi)-\varepsilon(\triangle u_{0},\varphi)+\int_{0}^{t}(\triangle u(s),\varphi)ds-\int_{0}^{t}(f(s),\varphi)ds+\int_{0}^{t}(g(s),\varphi)dB(s)
\end{array}
\end{array}
\end{equation}
for all $t\in [0,T]$ and all $\varphi\in L^{2}(G)$, for almost all $\omega\in \Omega.$
\end{proposition}
\begin{proof}
The main idea in this part comes from \cite{K1,G2,G3,G4}.
\par
We consider the stochastic differential equation
\begin{equation}\label{2}
\begin{array}{l}
\left\{
\begin{array}{llll}
(1+\varepsilon\lambda_{k})dc_{k}+(\lambda_{k}c_{k}+f_{k})dt=g_{k}dB
\\
c_{k}(0)= (u_{0},e_{k}),

\end{array}
\right.
\end{array}
\end{equation}
where
\begin{equation*}
\begin{array}{l}
\begin{array}{llll}
f_{k}(t)=(f(t),e_{k}),~g_{k}(t)=(g(t),e_{k}).
\end{array}
\end{array}
\end{equation*}
We set
\begin{equation*}
\begin{array}{l}
\begin{array}{llll}
u^{m}=\sum_{k=1}^{m}c_{k}(t)e_{k},\\
u_{0m}=\sum_{k=1}^{m}c_{k}(0)e_{k}=\sum_{k=1}^{m}(u_{0},e_{k})e_{k},\\
f^{m}=\sum_{k=1}^{m}c_{k}(t)e_{k},\\
g^{m}=\sum_{k=1}^{m}c_{k}(t)e_{k}
,
\end{array}
\end{array}
\end{equation*}
it holds that
\begin{equation*}
\begin{array}{l}
\begin{array}{llll}
\|u_{0m}-u_{0}\|_{L^{2}(\Omega;L^{2}(G))}\rightarrow0,\\
\|f^{m}-f\|_{L^{2}(\Omega;L^{2}(0,T;H^{-1}(G)))}\rightarrow0,\\
\|g^{m}-g\|_{L^{2}(\Omega;L^{2}(0,T;L^{2}(G)))}\rightarrow0
,
\end{array}
\end{array}
\end{equation*}
as $m\rightarrow\infty.$
\par
1) We have
\begin{equation*}
\begin{array}{l}
\begin{array}{llll}
\|u^{m}(t)\|_{L^{2}(G)}^{2}=\sum_{k=1}^{m}c_{k}^{2}(t),
\end{array}
\end{array}
\end{equation*}
it follows from It\^{o}'s rule that
\begin{equation*}
\begin{array}{l}
\begin{array}{llll}
dc_{k}^{2}=2c_{k}dc_{k}+(dc_{k})^{2}
\\~~~~=2c_{k}\frac{1}{1+\varepsilon\lambda_{k}}(-\lambda_{k}c_{k}dt-f_{k}dt+g_{k}dB)+\frac{1}{(1+\varepsilon\lambda_{k})^{2}}g_{k}^{2}dt
\\~~~~=-\frac{2\lambda_{k}c_{k}^{2}}{1+\varepsilon\lambda_{k}}dt-\frac{2c_{k}f_{k}}{1+\varepsilon\lambda_{k}}dt+\frac{2 c_{k}g_{k}}{1+\varepsilon\lambda_{k}}dB+\frac{1}{(1+\varepsilon\lambda_{k})^{2}}g_{k}^{2}dt,
\end{array}
\end{array}
\end{equation*}
thus,
\begin{equation*}
\begin{array}{l}
\begin{array}{llll}
c_{k}^{2}(t)+\int_{0}^{t}\frac{2\lambda_{k}c_{k}^{2}}{1+\varepsilon\lambda_{k}}ds
\\=c_{k}^{2}(0)-\int_{0}^{t}\frac{2c_{k}f_{k}}{1+\varepsilon\lambda_{k}}ds+\int_{0}^{t}\frac{2c_{k}g_{k}}{1+\varepsilon\lambda_{k}}dB+\int_{0}^{t}\frac{1}{(1+\varepsilon\lambda_{k})^{2}}g_{k}^{2}ds
\\\leq c_{k}^{2}(0)+\int_{0}^{t}\frac{\lambda_{k}c_{k}^{2}}{1+\varepsilon\lambda_{k}}ds+\int_{0}^{t}\frac{f_{k}^{2}}{(1+\varepsilon\lambda_{k})\lambda_{k}}ds+\int_{0}^{t}\frac{2c_{k}g_{k}}{1+\varepsilon\lambda_{k}}dB+\int_{0}^{t}\frac{1}{(1+\varepsilon\lambda_{k})^{2}}g_{k}^{2}ds,
\end{array}
\end{array}
\end{equation*}
namely, we have
\begin{equation*}
\begin{array}{l}
\begin{array}{llll}
c_{k}^{2}(t)+\int_{0}^{t}\frac{\lambda_{k}c_{k}^{2}}{1+\varepsilon\lambda_{k}}ds
\leq c_{k}^{2}(0)+\int_{0}^{t}\frac{f_{k}^{2}}{(1+\varepsilon\lambda_{k})\lambda_{k}}ds+\int_{0}^{t}\frac{2c_{k}g_{k}}{1+\varepsilon\lambda_{k}}dB+\int_{0}^{t}\frac{1}{(1+\varepsilon\lambda_{k})^{2}}g_{k}^{2}ds.
\end{array}
\end{array}
\end{equation*}
Taking mathematical expectation from both sides of the above inequality, we have
\begin{equation}\label{27}
\begin{array}{l}
\begin{array}{llll}
\mathbb{E} \int_{0}^{T}\frac{\lambda_{k}c_{k}^{2}}{1+\varepsilon\lambda_{k}}dt
\leq \mathbb{E} c_{k}^{2}(0)+\mathbb{E}\int_{0}^{T}\frac{f_{k}^{2}}{(1+\varepsilon\lambda_{k})\lambda_{k}}dt+\mathbb{E}\int_{0}^{T}\frac{1}{(1+\varepsilon\lambda_{k})^{2}}g_{k}^{2}dt
.
\end{array}
\end{array}
\end{equation}
By the Burkholder-Davis-Gundy inequality, we have
\begin{equation*}
\begin{array}{l}
\begin{array}{llll}
\mathbb{E} \sup\limits_{0\leq t\leq T}c_{k}^{2}(t)
\\\leq \mathbb{E}c_{k}^{2}(0)+\mathbb{E} \int_{0}^{T}\frac{f_{k}^{2}}{(1+\varepsilon\lambda_{k})\lambda_{k}}dt+\mathbb{E} \sup\limits_{0\leq t\leq T}|\int_{0}^{t}\frac{2c_{k}g_{k}}{1+\varepsilon\lambda_{k}}dB|+\mathbb{E} \int_{0}^{T}\frac{1}{(1+\varepsilon\lambda_{k})^{2}}g_{k}^{2}dt
\\\leq \mathbb{E}c_{k}^{2}(0)+\mathbb{E} \int_{0}^{T}\frac{f_{k}^{2}}{(1+\varepsilon\lambda_{k})\lambda_{k}}dt+\frac{1}{2}\mathbb{E} \sup\limits_{0\leq t\leq T}c_{k}^{2}(t)+C\mathbb{E} \int_{0}^{T}\frac{1}{(1+\varepsilon\lambda_{k})^{2}}g_{k}^{2}dt+\mathbb{E} \int_{0}^{T}\frac{1}{(1+\varepsilon\lambda_{k})^{2}}g_{k}^{2}dt,
\end{array}
\end{array}
\end{equation*}
thus,
\begin{equation}\label{22}
\begin{array}{l}
\begin{array}{llll}
\mathbb{E} \sup\limits_{0\leq t\leq T}c_{k}^{2}(t)
\leq C(\mathbb{E}c_{k}^{2}(0)+\mathbb{E} \int_{0}^{T}\frac{f_{k}^{2}}{(1+\varepsilon\lambda_{k})\lambda_{k}}dt+\mathbb{E} \int_{0}^{T}\frac{1}{(1+\varepsilon\lambda_{k})^{2}}g_{k}^{2}dt).
\end{array}
\end{array}
\end{equation}
According to (\ref{27}) and (\ref{22}), we have
\begin{equation}\label{28}
\begin{array}{l}
\begin{array}{llll}
\mathbb{E} \sup\limits_{0\leq t\leq T}c_{k}^{2}(t)+\mathbb{E} \int_{0}^{T}\frac{\lambda_{k}c_{k}^{2}}{1+\varepsilon\lambda_{k}}dt
\leq C(\mathbb{E}c_{k}^{2}(0)+\mathbb{E} \int_{0}^{T}\frac{f_{k}^{2}}{(1+\varepsilon\lambda_{k})\lambda_{k}}dt+\mathbb{E} \int_{0}^{T}\frac{1}{(1+\varepsilon\lambda_{k})^{2}}g_{k}^{2}dt).
\end{array}
\end{array}
\end{equation}
Taking the sum on $k$ in (\ref{28}), we get
\begin{equation}\label{31}
\begin{array}{l}
\begin{array}{llll}
\mathbb{E} \sup\limits_{0\leq t\leq T}\|u^{m}(t)\|_{L^{2}(G)}^{2}
\leq C[\mathbb{E}\|u_{0m}\|_{L^{2}(G)}^{2}+\mathbb{E} \int_{0}^{T}\|f^{m}(t)\|_{H^{-1}(G)}^{2}dt+\mathbb{E} \int_{0}^{T}\|g^{m}(t)\|_{L^{2}(G)}^{2}dt]
\end{array}
\end{array}
\end{equation}
thus,
\begin{equation}\label{23}
\begin{array}{l}
\begin{array}{llll}
\mathbb{E} \sup\limits_{0\leq t\leq T}\|(u^{m}-u^{n})(t)\|_{L^{2}(G)}^{2}\\
\leq C[\mathbb{E}\|u_{0m}-u_{0n}\|_{L^{2}(G)}^{2}+\mathbb{E} \int_{0}^{T}\|(f^{m}-f^{n})(t)\|_{H^{-1}(G)}^{2}dt+\mathbb{E} \int_{0}^{T}\|(g^{m}-g^{n})(t)\|_{L^{2}(G)}^{2}dt],
\end{array}
\end{array}
\end{equation}
where $C$ denotes a positive constant independent of $n,m$ and $T.$
\par
Next we observe that
the right-hand side of (\ref{23}) converges to zero as $n,m\rightarrow \infty$. Hence, it follows that
$\{u^{m}\}_{m=1}^{+\infty}$ is a Cauchy sequence that converges strongly in $L^{2}(\Omega,C([0,T];L^{2}(G)))$. Let
$u$ be the limit, namely, we have
\begin{equation*}
\begin{array}{l}
\begin{array}{llll}
\|u^{m}-u\|_{L^{2}(\Omega,C([0,T];L^{2}(G)))}\rightarrow0,
\end{array}
\end{array}
\end{equation*}
as $m\rightarrow\infty.$
\par
Also, it follows from (\ref{2}) that
\begin{equation*}
\begin{array}{l}
\begin{array}{llll}
(u^{m}(t),e_{k})-\varepsilon(u^{m}(t),\triangle e_{k})
\\~~~~~~~~~~~=(u_{0m},e_{k})-\varepsilon(u_{0m},\triangle e_{k})+\int_{0}^{t}( u^{m}(s),\triangle e_{k})ds-\int_{0}^{t}(f^{m}(s),e_{k})ds+\int_{0}^{t}(g^{m}(s),e_{k})dB(s)
\end{array}
\end{array}
\end{equation*}
for all $k=1,2,3\cdots$, and all $t\in [0,T]$, for almost all $\omega\in \Omega$.
\par
By taking the limit in above equality as $m$ goes to infinity, it holds that
\begin{equation*}
\begin{array}{l}
\begin{array}{llll}
(u(t),e_{k})-\varepsilon(u(t),\triangle e_{k})
\\~~~~~~~~~~~=(u_{0},e_{k})-\varepsilon(u_{0},\triangle e_{k})+\int_{0}^{t}( u(s),\triangle e_{k})ds-\int_{0}^{t}(f(s),e_{k})ds+\int_{0}^{t}(g(s),e_{k})dB(s)
\end{array}
\end{array}
\end{equation*}
for all $k=1,2,3\cdots$, and all $t\in [0,T]$, for almost all $\omega\in \Omega$.
Thus, we have
\begin{equation*}
\begin{array}{l}
\begin{array}{llll}
(u(t),\varphi)-\varepsilon(u(t),\triangle \varphi)
\\~~~~~~~~~~~=(u_{0},\varphi)-\varepsilon(u_{0},\triangle\varphi)+\int_{0}^{t}( u(s),\triangle\varphi)ds-\int_{0}^{t}(f(s),\varphi)ds+\int_{0}^{t}(g(s),\varphi)dB(s)
\end{array}
\end{array}
\end{equation*}
holds for all $t\in [0,T]$ and all $\varphi\in H^{2}(G)\cap H^{1}_{0}(G)$, for almost all $\omega\in \Omega.$
\par
Namely, $u$ is a solution to (\ref{1}).
By taking the limit in (\ref{31}) as $m$ goes to infinity, we can obtain (\ref{24}).
\par
Now, we prove the uniqueness of the solution for (\ref{1}).
Indeed, if $u_{1}$ and $u_{2}$ are the solutions for (\ref{1}), according to (\ref{24}), we have
\begin{equation*}
\begin{array}{l}
\begin{array}{llll}
\mathbb{E} \sup\limits_{0\leq t\leq T}\|(u_{1}-u_{2})(t)\|_{L^{2}(G)}^{2}\leq 0,
\end{array}
\end{array}
\end{equation*}
thus,
\begin{equation*}
\begin{array}{l}
\begin{array}{llll}
u_{1}\equiv u_{2}.
\end{array}
\end{array}
\end{equation*}
\par
2) Let
\begin{equation*}
h_{k}=(1+\varepsilon\lambda_{k})c_{k}^{2},
\end{equation*}
following \cite[P28]{K2} or \cite{P2}, we have
\begin{equation*}
\begin{array}{l}
\begin{array}{llll}
\|u^{m}(t)\|_{L^{2}(G)}^{2}+\varepsilon\|\nabla u^{m}(t)\|_{L^{2}(G)}^{2}=\sum_{k=1}^{m}(1+\varepsilon \lambda_{k})c_{k}^{2}(t)=\sum_{k=1}^{m}h_{k}.
\end{array}
\end{array}
\end{equation*}
\par
By multiplying (\ref{28}) by $1+\varepsilon \lambda_{k}$, we have
\begin{equation}\label{30}
\begin{array}{l}
\begin{array}{llll}
\mathbb{E} \sup\limits_{0\leq t\leq T}h_{k}(t)+\mathbb{E} \int_{0}^{T}\lambda_{k}c_{k}^{2}dt
\leq C(\mathbb{E}h_{k}(0)+\mathbb{E} \int_{0}^{T}\frac{f_{k}^{2}}{\lambda_{k}}dt+\mathbb{E} \int_{0}^{T}\frac{1}{1+\varepsilon \lambda_{k}}g_{k}^{2}dt)
\\
\leq C(\mathbb{E}h_{k}(0)+\mathbb{E} \int_{0}^{T}\frac{f_{k}^{2}}{\lambda_{k}}dt+\mathbb{E} \int_{0}^{T}g_{k}^{2}dt)

.
\end{array}
\end{array}
\end{equation}
Taking the sum on $k$ in (\ref{30}), we get
\begin{equation}\label{5}
\begin{array}{l}
\begin{array}{llll}
\mathbb{E} \sup\limits_{0\leq t\leq T}(\|u^{m}(t)\|_{L^{2}(G)}^{2}+\varepsilon\|\nabla u^{m}(t)\|_{L^{2}(G)}^{2})+\mathbb{E} \int_{0}^{T}\|\nabla u^{m}(t)\|_{L^{2}(G)}^{2}dt
\\\leq C[\mathbb{E}(\|u_{0m}\|_{L^{2}(G)}^{2}+\varepsilon\|\nabla u_{0m}\|_{L^{2}(G)}^{2})+\mathbb{E} \int_{0}^{T}\|f^{m}(t)\|_{H^{-1}(G)}^{2}dt+\mathbb{E} \int_{0}^{T}\|g^{m}(t)\|_{L^{2}(G)}^{2}dt],
\end{array}
\end{array}
\end{equation}
thus,
\begin{equation}\label{3}
\begin{array}{l}
\begin{array}{llll}
\mathbb{E} \sup\limits_{0\leq t\leq T}(\|(u^{m}-u^{n})(t)\|_{L^{2}(G)}^{2}+\varepsilon\|\nabla (u^{m}-u^{n})(t)\|_{L^{2}(G)}^{2})
\\+\mathbb{E} \int_{0}^{T}\|\nabla (u^{m}-u^{n})(t)\|_{L^{2}(G)}^{2}dt
\\\leq C[\mathbb{E}(\|u_{0m}-u_{0n}\|_{L^{2}(G)}^{2}+\varepsilon\|\nabla u_{0m}-\nabla u_{0n}\|_{L^{2}(G)}^{2})
\\+\mathbb{E} \int_{0}^{T}\|(f^{m}-f^{n})(t)\|_{H^{-1}(G)}^{2}dt+\mathbb{E} \int_{0}^{T}\|(g^{m}-g^{n})(t)\|_{L^{2}(G)}^{2}dt],
\end{array}
\end{array}
\end{equation}
where $C$ denotes a positive constant independent of $n,m$ and $T.$ Next we observe that
the right-hand side of (\ref{3}) converges to zero as $n,m\rightarrow \infty$. Hence, it follows that
$\{u^{m}\}_{m=1}^{+\infty}$ is a Cauchy sequence that converges strongly in $L^{2}(\Omega,C([0,T];H^{1}(G)))\cap L^{2}(\Omega,L^{2}(0,T;H^{1}(G)))$. Let
$u$ be the limit, namely, we have
\begin{equation*}
\begin{array}{l}
\begin{array}{llll}
\|u^{m}-u\|_{L^{2}(\Omega,C([0,T];H^{1}(G)))\bigcap L^{2}(\Omega,L^{2}(0,T;H^{1}(G)))}\rightarrow0,
\end{array}
\end{array}
\end{equation*}
as $m\rightarrow\infty.$
\par
Also, it follows from (\ref{2}) that
\begin{equation*}
\begin{array}{l}
\begin{array}{llll}
(u^{m}(t),e_{k})+\varepsilon(\nabla u^{m}(t),\nabla e_{k})
\\~~~~~~~~~~~=(u_{0m},e_{k})+\varepsilon(\nabla u_{0m},\nabla e_{k})-\int_{0}^{t}(\nabla u^{m}(s),\nabla e_{k})ds+\int_{0}^{t}(f^{m}(s),e_{k})ds+\int_{0}^{t}(g^{m}(s),e_{k})dB(s)
\end{array}
\end{array}
\end{equation*}
for all $k=1,2,3\cdots$, and all $t\in [0,T]$, for almost all $\omega\in \Omega$.
\par
By taking the limit in above equality as $m$ goes to infinity, it holds that
\begin{equation*}
\begin{array}{l}
\begin{array}{llll}
(u(t),e_{k})+\varepsilon(\nabla u(t),\nabla e_{k})
\\~~~~~~~~~~~=(u_{0},e_{k})+\varepsilon(\nabla u_{0},\nabla e_{k})-\int_{0}^{t}(\nabla u(s),\nabla e_{k})ds+\int_{0}^{t}(f(s),e_{k})ds+\int_{0}^{t}(g(s),e_{k})dB(s)
\end{array}
\end{array}
\end{equation*}
for all $k=1,2,3\cdots$, and all $t\in [0,T]$, for almost all $\omega\in \Omega$.
\par
Thus, it holds that
\begin{equation*}
\begin{array}{l}
\begin{array}{llll}
(u(t),\varphi)+\varepsilon(\nabla u(t),\nabla \varphi)
\\~~~~~~~~~~~=(u_{0},\varphi)+\varepsilon(\nabla u_{0},\nabla\varphi)+\int_{0}^{t}(\nabla u(s),\nabla\varphi)ds+\int_{0}^{t}(f(s),\varphi)ds+\int_{0}^{t}(g(s),\varphi)dB(s)
\end{array}
\end{array}
\end{equation*}
holds for all $t\in [0,T]$ and all $\varphi\in H^{1}_{0}(G)$, for almost all $\omega\in \Omega.$
\par
By taking the limit in (\ref{5}) as $m$ goes to infinity, we can obtain (\ref{29}).
\par
3) We have
\begin{equation*}
\begin{array}{l}
\begin{array}{llll}
\|\nabla u^{m}(t)\|_{L^{2}(G)}^{2}+\varepsilon\|\triangle u^{m}(t)\|_{L^{2}(G)}^{2}=\sum_{k=1}^{m}(\lambda_{k}+\varepsilon \lambda_{k}^{2})c_{k}^{2}(t)=\sum_{k=1}^{m}\lambda_{k}h_{k}.
\end{array}
\end{array}
\end{equation*}
\par
Multiplying (\ref{28}) by $(1+\varepsilon\lambda_{k})\lambda_{k}$, we have
\begin{equation}\label{6}
\begin{array}{l}
\begin{array}{llll}
\mathbb{E} \sup\limits_{0\leq t\leq T}(\lambda_{k}h_{k}(t))+\mathbb{E} \int_{0}^{T}\lambda_{k}^{2}c_{k}^{2}dt
\leq C(\mathbb{E}(\lambda_{k}h_{k}(0))+\mathbb{E} \int_{0}^{T}f_{k}^{2}dt+\mathbb{E} \int_{0}^{T}\lambda_{k}g_{k}^{2}dt).
\end{array}
\end{array}
\end{equation}
Taking the sum on $k$ in (\ref{6}), we get
\begin{equation*}
\begin{array}{l}
\begin{array}{llll}
\mathbb{E} \sup\limits_{0\leq t\leq T}(\|\nabla u^{m}(t)\|_{L^{2}(G)}^{2}+\varepsilon\|\triangle u^{m}(t)\|_{L^{2}(G)}^{2})+\mathbb{E} \int_{0}^{T}\|\triangle u^{m}(t)\|_{L^{2}(G)}^{2}dt
\\\leq C[\mathbb{E}(\|\nabla u_{0}\|_{L^{2}(G)}^{2}+\varepsilon\|\triangle u_{0}\|_{L^{2}(G)}^{2})+\mathbb{E} \int_{0}^{T}\|f^{m}(t)\|_{L^{2}(G)}^{2}dt+\mathbb{E} \int_{0}^{T}\|g^{m}(t)\|_{H^{1}(G)}^{2}dt],
\end{array}
\end{array}
\end{equation*}
thus,
\begin{equation}\label{7}
\begin{array}{l}
\begin{array}{llll}
\mathbb{E} \sup\limits_{0\leq t\leq T}(\|\nabla (u^{m}-u^{n})(t)\|_{L^{2}(G)}^{2}+\varepsilon\|\triangle (u^{m}-u^{n})(t)\|_{L^{2}(G)}^{2})
\\+\mathbb{E} \int_{0}^{T}\|\triangle (u^{m}-u^{n})(t)\|_{L^{2}(G)}^{2}dt
\\\leq C[\mathbb{E}(\|\nabla u_{0m}-\nabla u_{0n}\|_{L^{2}(G)}^{2}+\varepsilon\|\triangle u_{0m}-\triangle u_{0n}\|_{L^{2}(G)}^{2})
\\+\mathbb{E} \int_{0}^{T}\|(f^{m}-f^{n})(t)\|_{L^{2}(G)}^{2}dt+\mathbb{E} \int_{0}^{T}\|(g^{m}-g^{n})(t)\|_{H^{1}(G)}^{2}dt].
\end{array}
\end{array}
\end{equation}
where $C$ denotes a positive constant independent of $n,m$ and $T.$ Next we observe that
the right-hand side of (\ref{7}) converges to zero as $n,m\rightarrow \infty$. Hence, it follows that
$\{u^{m}\}_{m=1}^{+\infty}$ is a Cauchy sequence that converges strongly in $L^{2}(\Omega,C([0,T];H^{2}(G)))\cap L^{2}(\Omega,L^{2}(0,T;H^{2}(G)))$. Let
$u$ be the limit.
\par
By the same argument as in 1) and 2), $u$ is the solution of (\ref{1}).
\end{proof}

\section{Proof of Theorem \ref{Th3}}
\par
If there is no danger of confusion, we shall omit the subscript $\varepsilon,$ we use $\overline{u}_{n}$ instead of $\overline{u}_{n}^{\varepsilon}$ and $\overline{v}_{n}$ instead of $\overline{v}_{n}^{\varepsilon}.$
\par
The proof of the existence of the weak martingale solution is divided into several steps.
\par \textbf{Step 1. Construct the approximate solution.}

\par
Let $\{(\overline{\Omega},\overline{\mathcal{F}},\overline{\mathbb{P}}),(\overline{\mathcal{F}}_{t})_{0\leq t\leq T},\overline{B}\}$ be a fixed stochastic basis and $\{e_{n}:n=1,2,3\cdots\}$ be an orthonormal basis of $L^{2}(I)$ which was obtained in Lemma \ref{L6}. Set $H_{n}=Span\{e_{1},e_{2},...,e_{n}\}$ and let $P_{n}$ be the $L^{2}-$orthogonal projection from $L^{2}(I)$ onto $H_{n}.$
\par
We set
\begin{equation*}
\begin{array}{l}
\begin{array}{llll}
\overline{u}_{n}(t)=\sum_{k=1}^{n}c^{n}_{k}(t)e_{k}
\end{array}
\end{array}
\end{equation*}
and it is the solution of the following system of stochastic differential equations
\begin{eqnarray*}
\begin{array}{l}
\left\{
\begin{array}{llll}
d(\overline{u}_{n}-\varepsilon\overline{u}_{nxx})+(-\overline{u}_{nxx}+P_{n}\overline{u}_{n}^{3}-\overline{u}_{n})dt
=P_{n}g(\overline{u}_{n})d\overline{B}
\\\overline{u}_{n}(0,t)=0=\overline{u}_{n}(1,t),
\\\overline{u}_{n}(x,0)=P_{n}u_{0}\triangleq u_{n0}(x)
\end{array}
\right.
\end{array}
\begin{array}{lll}
{\rm{in}}~Q\\
{\rm{in}}~(0,T)\\
{\rm{in}}~I
\end{array}
\end{eqnarray*}
defined on $\{(\overline{\Omega},\overline{\mathcal{F}},\overline{\mathbb{P}}),(\overline{\mathcal{F}}_{t})_{0\leq t\leq T},\overline{B}\}.$ The mathematical expectation with respect to $\overline{\mathbb{P}}$ is denoted by $\overline{\mathbb{E}}.$
\par
It is easy to see that $c^{n}_{k}$ satisfies the following system of stochastic differential equations
\begin{eqnarray}\label{42}
\begin{array}{l}
\left\{
\begin{array}{llll}
dc^{n}_{k}+\frac{1}{1+\varepsilon \lambda_{k}}(\lambda_{k}c^{n}_{k}+(P_{n}\overline{u}_{n}^{3},e_{k})-c^{n}_{k})dt
=\frac{1}{1+\varepsilon \lambda_{k}}(P_{n}g(\overline{u}_{n}),e_{k})d\overline{B}
\\c^{n}_{k}(0)=(u_{0},e_{k}).
\end{array}
\right.
\end{array}
\end{eqnarray}
\par
By the theory of stochastic differential equations, there is a local $\overline{u}_{n}$ defined on $[0,T_{n}].$ The following a priori estimates will enable us to prove that $T_{n}=T.$

\par \textbf{Step 2. A priori estimates.}

\begin{lemma}
There exists a positive constant $C$ independent of $\varepsilon$ such that
\begin{eqnarray}\label{43}
\begin{array}{l}
\begin{array}{llll}
\overline{\mathbb{E}} \sup\limits_{0\leq t\leq T}(\|\overline{u}_{n}(t)\|_{L^{2}(I)}^{2}+\varepsilon\|\overline{u}_{nx}(t)\|_{L^{2}(I)}^{2})
+\overline{\mathbb{E}}\int_{0}^{T}(\|\overline{u}_{nx}(t)\|_{L^{2}(I)}^{2}+\|\overline{u}_{n}(t)\|_{L^{4}(I)}^{4})dt\leq C
\end{array}
\end{array}
\end{eqnarray}
for any $n\geq1.$
\end{lemma}
\begin{proof}
\par
Indeed, it follows from It\^{o}'s rule that
\begin{equation*}
\begin{array}{l}
\begin{array}{llll}
dc_{k}^{n2}=2c_{k}^{n}dc_{k}^{n}+(dc_{k}^{n})^{2}
\\~~~~=2c_{k}^{n}\frac{1}{1+\varepsilon\lambda_{k}}[(-\lambda_{k}c^{n}_{k}-(P_{n}\overline{u}_{n}^{3},e_{k})+c^{n}_{k})dt+(P_{n}g(\overline{u}_{n}),e_{k})d\overline{B}]+\frac{1}{(1+\varepsilon\lambda_{k})^{2}}|(P_{n}g(\overline{u}_{n}),e_{k})|^{2}dt
,
\end{array}
\end{array}
\end{equation*}
namely, we have
\begin{equation}\label{32}
\begin{array}{l}
\begin{array}{llll}
(1+\varepsilon\lambda_{k})dc_{k}^{n2}
\\~~~~=[(-2\lambda_{k}c^{n2}_{k}-2(P_{n}\overline{u}_{n}^{3},c_{k}^{n}e_{k})+2c^{n2}_{k})dt+2(P_{n}g(\overline{u}_{n}),c_{k}^{n}e_{k})d\overline{B}]+\frac{1}{1+\varepsilon\lambda_{k}}|(P_{n}g(\overline{u}_{n}),e_{k})|^{2}dt
.
\end{array}
\end{array}
\end{equation}
Taking the sum on $k$ in (\ref{32}), following \cite[P28]{K2} or \cite{P2}, we get
\begin{eqnarray}\label{40}
\begin{array}{l}
\begin{array}{llll}
d(\|\overline{u}_{n}(t)\|_{L^{2}(I)}^{2}+\varepsilon\|\overline{u}_{nx}(t)\|_{L^{2}(I)}^{2})+2(\|\overline{u}_{nx}\|_{L^{2}(I)}^{2}+\|\overline{u}_{n}\|_{L^{4}(I)}^{4})dt
\\=(2\|\overline{u}_{n}\|_{L^{2}(I)}^{2}
+\sum_{k=1}^{n}\frac{1}{1+\varepsilon\lambda_{k}}|(P_{n}g(\overline{u}_{n}),e_{k})|^{2})dt
+2(\overline{u}_{n},P_{n}g(\overline{u}_{n}))d\overline{B},
\end{array}
\end{array}
\end{eqnarray}
namely,
\begin{eqnarray}\label{36}
\begin{array}{l}
\begin{array}{llll}
(\|\overline{u}_{n}(t)\|_{L^{2}(I)}^{2}+\varepsilon\|\overline{u}_{nx}(t)\|_{L^{2}(I)}^{2})+2\int_{0}^{t}(\|\overline{u}_{nx}\|_{L^{2}(I)}^{2}+\|\overline{u}_{n}\|_{L^{4}(I)}^{4})ds
\\=\|u_{n0}\|_{L^{2}(I)}^{2}+\varepsilon\|u_{n0x}\|_{L^{2}(I)}^{2}+\int_{0}^{t}\left(2\|\overline{u}_{n}\|_{L^{2}(I)}^{2}
+\sum_{k=1}^{n}\frac{1}{1+\varepsilon\lambda_{k}}|(P_{n}g(\overline{u}_{n}),e_{k})|^{2}\right)ds
\\+2\int_{0}^{t}(\overline{u}_{n},P_{n}g(\overline{u}_{n}))d\overline{B}.
\end{array}
\end{array}
\end{eqnarray}
It is easy to see
\begin{equation*}
\begin{array}{l}
\begin{array}{llll}
\overline{\mathbb{E}} \left|\int_{0}^{t}\sum_{k=1}^{n}\frac{1}{1+\varepsilon\lambda_{k}}|(P_{n}g(\overline{u}_{n}),e_{k})|^{2}ds\right|
\\\leq\overline{\mathbb{E}} \left|\int_{0}^{t}\sum_{k=1}^{n}|(P_{n}g(\overline{u}_{n}),e_{k})|^{2}ds\right|
\\\leq\overline{\mathbb{E}} \left|\int_{0}^{t}\|P_{n}g(\overline{u}_{n})\|_{L^{2}(I)}^{2}ds\right|
\\\leq C \overline{\mathbb{E}}\int_{0}^{t}\left(1+\|\overline{u}_{n}(s)\|_{L^{2}(I)}^{2}\right)ds.
\end{array}
\end{array}
\end{equation*}
\par

By the Burkholder-Davis-Gundy inequality and Cauchy inequality, we can obtain that for any $\delta>0$,
\begin{equation*}
\begin{array}{l}
\begin{array}{llll}
~~~\displaystyle \overline{\mathbb{E}} \sup\limits_{0\leq s\leq t}\left|\int_{0}^{s}(\overline{u}_{n},P_{n}g(\overline{u}_{n}))d\overline{B}\right|
\\ \displaystyle=\overline{E} \sup\limits_{0\leq s\leq t}\left|\int_{0}^{s}(P_{n}\overline{u}_{n},g(\overline{u}_{n}))d\overline{B}\right|
\\ \displaystyle=\overline{E} \sup\limits_{0\leq s\leq t}\left|\int_{0}^{s}(\overline{u}_{n},g(\overline{u}_{n}))d\overline{B}\right|
\\ \displaystyle\leq \delta \overline{\mathbb{E}} \sup\limits_{0\leq s\leq t} \|\overline{u}_{n}(s)\|_{L^{2}(I)}^{2}+C(\delta) \overline{\mathbb{E}}\int_{0}^{t}\|g(\overline{u}_{n})(s))\|_{L^{2}(I)}^{2}ds

\\ \displaystyle\leq \delta \overline{\mathbb{E}} \sup\limits_{0\leq s\leq t} \|\overline{u}_{n}(s)\|_{L^{2}(I)}^{2}+C(\delta) \overline{\mathbb{E}}\int_{0}^{t}\left(1+\|\overline{u}_{n}(s)\|_{L^{2}(I)}^{2}\right)ds.
\end{array}
\end{array}
\end{equation*}
It follows from (\ref{36}) that
\begin{equation*}
\begin{array}{l}
\begin{array}{llll}
\overline{\mathbb{E}} \sup\limits_{0\leq s\leq t} (\|\overline{u}_{n}(s)\|_{L^{2}(I)}^{2}+\varepsilon\|\overline{u}_{nx}(s)\|_{L^{2}(I)}^{2})+2\overline{\mathbb{E}}\int_{0}^{t}(\|\overline{u}_{nx}\|_{L^{2}(I)}^{2}
+\|\overline{u}_{n}\|_{L^{4}(I)}^{4})ds
\\\leq \displaystyle\delta \overline{\mathbb{E}} \sup\limits_{0\leq s\leq t} \|\overline{u}_{n}(s)\|_{L^{2}(I)}^{2}+C\left(\overline{\mathbb{E}}\|u_{n0}\|_{H^{1}(I)}^{2}+\overline{\mathbb{E}}\int_{0}^{t}\left(1+\|\overline{u}_{n}(s)\|_{L^{2}(I)}^{2}\right)ds\right).
\end{array}
\end{array}
\end{equation*}
By choosing $\delta>0$ small enough, yields
\begin{equation*}
\begin{array}{l}
\begin{array}{llll}
\displaystyle \overline{\mathbb{E}} \sup\limits_{0\leq s\leq t} (\|\overline{u}_{n}(s)\|_{L^{2}(I)}^{2}+\varepsilon\|\overline{u}_{nx}(s)\|_{L^{2}(I)}^{2})
+\overline{\mathbb{E}}\int_{0}^{t}(\|\overline{u}_{nx}\|_{L^{2}(I)}^{2}
+\|\overline{u}_{n}\|_{L^{4}(I)}^{4})ds
\\\leq \displaystyle C\left(\overline{\mathbb{E}}\|u_{n0}\|_{H^{1}(I)}^{2}+\overline{\mathbb{E}}\int_{0}^{t}\left(1+\|\overline{u}_{n}(s)\|_{L^{2}(I)}^{2}\right)ds\right).
\end{array}
\end{array}
\end{equation*}
According to Gronwall's lemma, we obtain that
\begin{equation*}
\begin{array}{l}
\begin{array}{llll}
\displaystyle \overline{\mathbb{E}} \sup\limits_{0\leq t\leq T} (\|\overline{u}_{n}(s)\|_{L^{2}(I)}^{2}+\varepsilon\|\overline{u}_{nx}(s)\|_{L^{2}(I)}^{2})\leq C,
\\ \displaystyle \overline{\mathbb{E}}\int_{0}^{T}(\|\overline{u}_{nx}\|_{L^{2}(I)}^{2}+\|\overline{u}_{n}\|_{L^{4}(I)}^{4})dt\leq C.
\end{array}
\end{array}
\end{equation*}
\end{proof}
\par
The following result is related to the higher integrability of $\overline{u}_{n}.$
\begin{lemma}\label{L2}
For any $1\leq p <\infty,$ there exists a constant $C_{p}$ independent of $\varepsilon$ such that
\begin{eqnarray}
\label{44}
&&\overline{\mathbb{E}} \sup\limits_{0\leq t\leq T}(\|\overline{u}_{n}(s)\|_{L^{2}(I)}^{2}+\varepsilon\|\overline{u}_{nx}(s)\|_{L^{2}(I)}^{2})^{\frac{p}{2}}\leq C_{p},
\\ \label{45} &&\overline{\mathbb{E}}\left(\int_{0}^{T}(\|\overline{u}_{nx}\|_{L^{2}(I)}^{2}+\|\overline{u}_{n}\|_{L^{4}(I)}^{4})dt\right)^{\frac{p}{2}}\leq C_{p}
\end{eqnarray}
for any $n\geq1.$
\end{lemma}
\begin{proof}
Case I: $2\leq p <\infty.$
\par
To simplify the notation, we define
\begin{equation*}
\begin{array}{l}
\begin{array}{llll}
\phi_{n}=\|\overline{u}_{n}(t)\|_{L^{2}(I)}^{2}+\varepsilon\|\overline{u}_{nx}(t)\|_{L^{2}(I)}^{2},\\
K=(2\|\overline{u}_{n}\|_{L^{2}(I)}^{2}
+\sum_{k=1}^{n}\frac{1}{1+\varepsilon\lambda_{k}}|(P_{n}g(t,\overline{u}_{n}),e_{k})|^{2})-2(\|\overline{u}_{nx}\|_{L^{2}(I)}^{2}+\|\overline{u}_{n}\|_{L^{4}(I)}^{4}),
\\L=2(\overline{u}_{n},P_{n}g(t,\overline{u}_{n})).
\end{array}
\end{array}
\end{equation*}
Thus we can rewrite (\ref{40}) as
\begin{equation*}
d\phi_{n}=Kdt+Ld\overline{B}.
\end{equation*}
By It\^{o}'s rule, we obtain that
\begin{equation*}
d\phi_{n}^{\frac{p}{2}}=\frac{p}{2}\phi_{n}^{\frac{p-2}{2}}\left((K+\frac{p-2}{4}\phi_{n}^{-1}L^{2})dt+Ld\overline{B}\right),
\end{equation*}
for any $2\leq p <\infty.$ Namely, we have
\begin{equation}\label{39}
\phi_{n}^{\frac{p}{2}}(t)=\phi_{n}^{\frac{p}{2}}(0)+\int_{0}^{t}\frac{p}{2}\phi_{n}^{\frac{p-2}{2}}(K+\frac{p-2}{4}\phi_{n}^{-1}L^{2})ds
+\int_{0}^{t}\frac{p}{2}\phi_{n}^{\frac{p-2}{2}}Ld\overline{B}.
\end{equation}
\par
Using the properties of $g$ and Young's inequality, we have
\begin{equation*}
\begin{array}{l}
\begin{array}{llll}
\phi_{n}^{\frac{p-2}{2}}K
\\\leq \phi_{n}^{\frac{p-2}{2}}(2\|\overline{u}_{n}\|_{L^{2}(I)}^{2}
+\sum_{k=1}^{n}\frac{1}{1+\varepsilon\lambda_{k}}|(P_{n}g(\overline{u}_{n}),e_{k})|^{2}-2(\|\overline{u}_{nx}\|_{L^{2}(I)}^{2}+\|\overline{u}_{n}\|_{L^{4}(I)}^{4})
)
\\\leq \phi_{n}^{\frac{p-2}{2}}(2\|\overline{u}_{n}\|_{L^{2}(I)}^{2}
+\sum_{k=1}^{n}\frac{1}{1+\varepsilon\lambda_{k}}|(P_{n}g(\overline{u}_{n}),e_{k})|^{2})
\\\leq \phi_{n}^{\frac{p-2}{2}}(2\|\overline{u}_{n}\|_{L^{2}(I)}^{2}
+\|P_{n}g(\overline{u}_{n})\|_{L^{2}(I)}^{2})
\\\leq C\phi_{n}^{\frac{p-2}{2}}(1+\|\overline{u}_{n}\|_{L^{2}(I)}^{2}
)

\\\leq C(1+\phi_{n}^{\frac{p}{2}}),
\\

\\
\phi_{n}^{\frac{p-2}{2}}\phi_{n}^{-1}L^{2}
\\=C\phi_{n}^{\frac{p-4}{2}}(\overline{u}_{n},P_{n}g(\overline{u}_{n}))^{2}
\\\leq C\phi_{n}^{\frac{p-4}{2}}\|\overline{u}_{n}\|_{L^{2}(I)}^{2}\|P_{n}g(\overline{u}_{n})\|_{L^{2}(I)}^{2}
\\\leq C\phi_{n}^{\frac{p-4}{2}}\|\overline{u}_{n}\|_{L^{2}(I)}^{2}(1+\|\overline{u}_{n}\|_{L^{2}(I)}^{2})
\\\leq C(1+\phi_{n}^{\frac{p}{2}}).

\end{array}
\end{array}
\end{equation*}
According to the Burkholder-Davis-Gundy inequality and Young's inequality, it can be deduced that
\begin{equation*}
\begin{array}{l}
\begin{array}{llll}
~~~\overline{\mathbb{E}}\sup\limits_{0\leq s\leq t}\left|\int_{0}^{s}\phi_{n}^{\frac{p-2}{2}}Ld\overline{B}\right|
\\=2\overline{\mathbb{E}}\sup\limits_{0\leq s\leq t}\left|\int_{0}^{s}\phi_{n}^{\frac{p-2}{2}}(\overline{u}_{n},P_{n}g(\overline{u}_{n}))d\overline{B}\right|
\\\leq C\overline{\mathbb{E}}\left(\int_{0}^{t}\phi_{n}^{p-2}(\overline{u}_{n},P_{n}g(\overline{u}_{n}))^{2}ds\right)^{\frac{1}{2}}
\\\leq C\overline{\mathbb{E}}\left(\int_{0}^{t}\phi_{n}^{p-2}\|\overline{u}_{n}\|_{L^{2}(I)}^{2}\|P_{n}g(\overline{u}_{n})\|_{L^{2}(I)}^{2}ds\right)^{\frac{1}{2}}
\\\leq C\overline{\mathbb{E}}\left(\int_{0}^{t}\phi_{n}^{p-2}\|\overline{u}_{n}\|_{L^{2}(I)}^{2}\|g(\overline{u}_{n})\|_{L^{2}(I)}^{2}ds\right)^{\frac{1}{2}}
\\\leq C\overline{\mathbb{E}}\left(\int_{0}^{t}\phi_{n}^{p-2}\|\overline{u}_{n}\|_{L^{2}(I)}^{2}(1+\|\overline{u}_{n}\|_{L^{2}(I)}^{2})ds\right)^{\frac{1}{2}}
\\\leq C\overline{\mathbb{E}}\left(\int_{0}^{t}(1+\phi_{n}^{p})ds\right)^{\frac{1}{2}}
\\\leq C+C\overline{\mathbb{E}}\left(\int_{0}^{t}\phi_{n}^{p}ds\right)^{\frac{1}{2}}
\\\leq C+C\overline{\mathbb{E}}\left(\int_{0}^{t}\phi_{n}^{\frac{p}{2}}\phi_{n}^{\frac{p}{2}}ds\right)^{\frac{1}{2}}
\\\leq C+C\overline{\mathbb{E}}\left(\sup\limits_{0\leq s\leq t}\phi_{n}^{\frac{p}{2}}\int_{0}^{t}\phi_{n}^{\frac{p}{2}}ds\right)^{\frac{1}{2}}
\\\leq C +\delta \overline{\mathbb{E}}\sup\limits_{0\leq s\leq t}\phi_{n}^{\frac{p}{2}}+C\overline{\mathbb{E}}\int_{0}^{t}\phi_{n}^{\frac{p}{2}}ds.
\end{array}
\end{array}
\end{equation*}
From the above estimates and (\ref{39}), by choosing $\delta>0$ small enough, it holds that
\begin{equation*}
\begin{array}{l}
\begin{array}{llll}
\displaystyle \overline{\mathbb{E}} \sup\limits_{0\leq s\leq t}\phi_{n}^{\frac{p}{2}}\leq C+C\overline{\mathbb{E}}\int_{0}^{t}\phi_{n}^{\frac{p}{2}}ds.
\end{array}
\end{array}
\end{equation*}
According to Gronwall's lemma and the definition of $\phi_{n}$, we obtain that
\begin{equation}
\label{37}\displaystyle \overline{\mathbb{E}} \sup\limits_{0\leq t\leq T}(\|\overline{u}_{n}(s)\|_{L^{2}(I)}^{2}+\varepsilon\|\overline{u}_{nx}(s)\|_{L^{2}(I)}^{2})^{\frac{p}{2}}\leq C_{p}.
\end{equation}

In view of (\ref{36}), there holds
\begin{equation*}
\begin{array}{l}
\begin{array}{llll}
(\|\overline{u}_{n}(t)\|_{L^{2}(I)}^{2}+\varepsilon\|\overline{u}_{nx}(t)\|_{L^{2}(I)}^{2})+2\int_{0}^{t}(\|\overline{u}_{nx}\|_{L^{2}(I)}^{2}+\|\overline{u}_{n}\|_{L^{4}(I)}^{4})ds
\\=\|u_{n0}\|_{L^{2}(I)}^{2}+\varepsilon\|u_{n0x}\|_{L^{2}(I)}^{2}+\int_{0}^{t}\left(2\|\overline{u}_{n}\|_{L^{2}(I)}^{2}
+\sum_{k=1}^{n}\frac{1}{1+\varepsilon\lambda_{k}}|(P_{n}g(\overline{u}_{n}),e_{k})|^{2}\right)ds
\\+2\int_{0}^{t}(\overline{u}_{n},P_{n}g(\overline{u}_{n}))d\overline{B}
\\\leq \|u_{n0}\|_{H^{1}(I)}^{2}+\int_{0}^{t}\left(2\|\overline{u}_{n}\|_{L^{2}(I)}^{2}
+\|P_{n}g(\overline{u}_{n})\|_{L^{2}(I)}^{2}\right)ds
\\+2\int_{0}^{t}(\overline{u}_{n},P_{n}g(\overline{u}_{n}))d\overline{B}

.
\end{array}
\end{array}
\end{equation*}
Thus, we have
\begin{equation*}
\begin{array}{l}
\begin{array}{llll}
\int_{0}^{T}\left(\|\overline{u}_{nx}\|_{L^{2}(I)}^{2}+\|\overline{u}_{n}\|_{L^{4}(I)}^{4}\right)dt
\\\leq C\left(\|u_{n0}\|_{H^{1}(I)}^{2}+\int_{0}^{T}(1+\|\overline{u}_{n}\|_{L^{2}(I)}^{2})dt
+\left|\int_{0}^{T}(\overline{u}_{n},P_{n}g(\overline{u}_{n}))d\overline{B}\right|\right),
\end{array}
\end{array}
\end{equation*}
then, for any $2\leq p <\infty,$ it holds that
\begin{equation*}
\begin{array}{l}
\begin{array}{llll}
\displaystyle~~~\left(\int_{0}^{T}\left(\|\overline{u}_{nx}\|_{L^{2}(I)}^{2}+\|\overline{u}_{n}\|_{L^{4}(I)}^{4}\right)ds\right)^{\frac{p}{2}}
\\ \displaystyle \leq C_{p}\left(\|u_{n0}\|_{H^{1}(I)}^{p}+\left(\int_{0}^{T}(1+\|\overline{u}_{n}(s)\|_{L^{2}(I)}^{2})ds\right)^{\frac{p}{2}}
+\left|\int_{0}^{T}(\overline{u}_{n},P_{n}g(\overline{u}_{n}))d\overline{B}\right|^{\frac{p}{2}}\right).
\end{array}
\end{array}
\end{equation*}
By the Burkholder-Davis-Gundy inequality and Young's inequality, we have
\begin{equation*}
\begin{array}{l}
\begin{array}{llll}
\displaystyle~~~\overline{\mathbb{E}}\left|\int_{0}^{T}(\overline{u}_{n},P_{n}g(\overline{u}_{n}))d\overline{B}\right|^{\frac{p}{2}}
\\\displaystyle\leq \overline{\mathbb{E}}\sup\limits_{0\leq t\leq T}\left|\int_{0}^{t}(\overline{u}_{n},P_{n}g(\overline{u}_{n}))d\overline{B}\right|^{\frac{p}{2}}
\\ \displaystyle\leq C \overline{\mathbb{E}}\left(\int_{0}^{T}(\overline{u}_{n},P_{n}g(\overline{u}_{n}))^{2}dt\right)^{\frac{p}{4}}
\\ \displaystyle\leq C \overline{\mathbb{E}}\left(\int_{0}^{T}(1+\phi_{n}^{2})dt\right)^{\frac{p}{4}}.
\end{array}
\end{array}
\end{equation*}
Thus
\begin{equation*}
\begin{array}{l}
\begin{array}{llll}
~~~\overline{\mathbb{E}}\left(\int_{0}^{T}\left(\|\overline{u}_{nx}\|_{L^{2}(I)}^{2}+\|\overline{u}_{n}\|_{L^{4}(I)}^{4}\right)dt\right)^{\frac{p}{2}}
\\\leq C\overline{\mathbb{E}}\|u_{n0}\|_{H^{1}(I)}^{p}
+C\overline{\mathbb{E}}\left(\int_{0}^{T}(1+\phi_{n})dt\right)^{\frac{p}{2}}
+C \overline{\mathbb{E}}\left(\int_{0}^{T}(1+\phi_{n}^{2})dt\right)^{\frac{p}{4}}
\\\leq C\overline{\mathbb{E}}\|u_{n0}\|_{H^{1}(I)}^{p}
+C\overline{\mathbb{E}}\left(\int_{0}^{T}(1+\sup\limits_{0\leq t\leq T}\phi_{n})dt\right)^{\frac{p}{2}}+C \overline{\mathbb{E}}\left(\int_{0}^{T}(1+\sup\limits_{0\leq t\leq T}\phi_{n}^{2})dt\right)^{\frac{p}{4}}
\\\leq C\overline{\mathbb{E}}\|u_{n0}\|_{H^{1}(I)}^{p}
+CT^{\frac{p}{2}}\overline{\mathbb{E}}\left(1+\sup\limits_{0\leq t\leq T}\phi_{n}\right)^{\frac{p}{2}}+C T^{\frac{p}{4}}\overline{\mathbb{E}}\left(1+\sup\limits_{0\leq t\leq T}\phi_{n}^{2}\right)^{\frac{p}{4}}
\\\leq C(1+\overline{\mathbb{E}}\|u_{n0}\|_{H^{1}(I)}^{p}
+\overline{\mathbb{E}}\sup\limits_{0\leq t\leq T}\phi_{n}^{\frac{p}{2}}).
\end{array}
\end{array}
\end{equation*}
According to (\ref{37}), it holds that
\begin{equation*}
\overline{\mathbb{E}}\left(\int_{0}^{T}(\|\overline{u}_{nx}\|_{L^{2}(I)}^{2}+\|\overline{u}_{n}\|_{L^{4}(I)}^{4})dt\right)^{\frac{p}{2}}\leq C_{p}.
\end{equation*}
\par
Case II: $1\leq p <2.$
\par
This case can be obtained from Case I and the Young inequality.
\end{proof}
\par
The next estimate is very important for the proof of the tightness of the law of the Galerkin
solution $\{\overline{u}_{n}\}_{n\geq 1}.$
\begin{lemma}\label{L3}
There exists a positive constant $C$ independent of $\varepsilon$ such that
\begin{equation}\label{46}
\overline{\mathbb{E}} \sup\limits_{0\leq |\theta|\leq \delta} \int_{0}^{T}\|\overline{u}_{n}(t+\theta)-\overline{u}_{n}(t)\|_{H^{-1}(I)}^{2}dt\leq C\delta,
\end{equation}
for any $0<\delta\leq1.$
\end{lemma}
\begin{remark}
In the above lemma, $\overline{u}_{n}$ is extended to $0$ outside $[0,T]$.
\end{remark}
\begin{proof}
We set
\begin{equation*}
\begin{array}{l}
\begin{array}{llll}
\overline{ v}_{n}(t)=(\overline{u}_{n}-\varepsilon \overline{u}_{nxx})(t),
\end{array}
\end{array}
\end{equation*}
it is easy to see that
\begin{equation*}
\begin{array}{l}
\begin{array}{llll}
\overline{ v}_{n}(t+\theta)-\overline{ v}_{n}(t)=\int_{t}^{t+\theta}\overline{u}_{nxx}(s)ds-\int_{t}^{t+\theta}(P_{n}\overline{u}_{n}^{3}-\overline{u}_{n})(s)ds+\int_{t}^{t+\theta}P_{n}g(\overline{u}_{n}(s))d\overline{B},
\end{array}
\end{array}
\end{equation*}
which implies
\begin{equation}\label{47}
\begin{array}{l}
\begin{array}{llll}
~~~\|\overline{ v}_{n}(t+\theta)-\overline{ v}_{n}(t)\|_{H^{-1}(I)}
\\\leq\|\int_{t}^{t+\theta}\overline{u}_{nxx}(s)ds\|_{H^{-1}(I)}
+\|\int_{t}^{t+\theta}(P_{n}\overline{u}^{3}_{n}-\overline{u}_{n})(s)ds\|_{H^{-1}(I)}+\|\int_{t}^{t+\theta}P_{n}g(\overline{u}_{n}(s))d\overline{B}\|_{H^{-1}(I)}
\\\leq\int_{t}^{t+\theta}\|\overline{u}_{nxx}(s)\|_{H^{-1}(I)}ds
+\int_{t}^{t+\theta}\|(P_{n}\overline{u}^{3}_{n}-\overline{u}_{n})(s)\|_{H^{-1}(I)}ds+\|\int_{t}^{t+\theta}P_{n}g(\overline{u}_{n}(s))d\overline{B}\|_{H^{-1}(I)}.
\end{array}
\end{array}
\end{equation}
Taking the square in both side of (\ref{47}), we have
\begin{equation*}
\begin{array}{l}
\begin{array}{llll}
\|\overline{ v}_{n}(t+\theta)-\overline{ v}_{n}(t)\|_{H^{-1}(I)}^{2}
\\\leq(\int_{t}^{t+\theta}\|\overline{u}_{nxx}(s)\|_{H^{-1}(I)}ds
+\int_{t}^{t+\theta}\|(P_{n}\overline{u}^{3}_{n}-\overline{u}_{n})(s)\|_{H^{-1}(I)}ds+\|\int_{t}^{t+\theta}P_{n}g(\overline{u}_{n}(s))d\overline{B}\|_{H^{-1}(I)})^{2}
\\\leq C\theta\int_{t}^{t+\theta}(\|\overline{u}_{nxx}(s)\|_{H^{-1}(I)}^{2}+\|(P_{n}\overline{u}^{3}_{n}-\overline{u}_{n})(s)\|_{H^{-1}(I)}^{2})ds+C\|\int_{t}^{t+\theta}P_{n}g(\overline{u}_{n}(s))d\overline{B}\|_{H^{-1}(I)}^{2}
.
\end{array}
\end{array}
\end{equation*}
\par
We can infer from (\ref{44}) and (\ref{45}) that
\begin{equation}\label{48}
\begin{array}{l}
\begin{array}{llll}
\mathbb{E}\int_{0}^{T}\int_{t}^{t+\delta}\|\overline{u}_{nxx}\|_{H^{-1}(I)}^{2}dsdt
\\\leq\delta \mathbb{E}\int_{0}^{T}\|\overline{u}_{nx}(t)\|_{L^{2}(I)}^{2}dt
\\\leq C\delta,
\\
\mathbb{E}\int_{0}^{T}\int_{t}^{t+\delta}\|(P_{n}\overline{u}^{3}_{n}-\overline{u}_{n})(s)\|_{H^{-1}(I)}^{2}dsdt
\\\leq
\mathbb{E}\int_{0}^{T}\int_{t}^{t+\delta}\|(P_{n}\overline{u}^{3}_{n}-\overline{u}_{n})(s)\|_{L^{2}(I)}^{2}dsdt
\\=\delta \mathbb{E}\int_{0}^{T}\|P_{n}\overline{u}^{3}_{n}-\overline{u}_{n}\|_{L^{2}(I)}^{2}dt
\\\leq C\delta [\mathbb{E}(\int_{0}^{T}\|\overline{u}_{nx}\|_{L^{2}(I)}^{2}dt)^{2}+\mathbb{E}\sup\limits_{0\leq t\leq T}\|\overline{u}_{n}\|_{L^{2}(I)}^{8}+\mathbb{E}\int_{0}^{T}\|\overline{u}_{n}\|_{L^{2}(I)}^{2}dt]
\\\leq C\delta.
\end{array}
\end{array}
\end{equation}

By the Burkholder-Davis-Gundy inequality and Young's inequality, we have
\begin{equation}\label{49}
\begin{array}{l}
\begin{array}{llll}
\mathbb{E} \sup\limits_{0\leq |\theta|\leq \delta} \int_{0}^{T}\|\int_{t}^{t+\theta}P_{n}g(\overline{u}_{n}(s))dB\|_{H^{-1}(I)}^{2}dt
\\\leq\mathbb{E} \sup\limits_{0\leq |\theta|\leq \delta} \int_{0}^{T}\|\int_{t}^{t+\theta}P_{n}g(\overline{u}_{n}(s))dB\|_{L^{2}(I)}^{2}dt
\\\leq \mathbb{E}\int_{0}^{T}\sup\limits_{0\leq |\theta|\leq \delta}\|\int_{t}^{t+\theta}P_{n}g(\overline{u}_{n}(s))dB\|_{L^{2}(I)}^{2}dt
\\=\int_{0}^{T}\mathbb{E}\sup\limits_{0\leq |\theta|\leq \delta}\|\int_{t}^{t+\theta}P_{n}g(\overline{u}_{n}(s))dB\|_{L^{2}(I)}^{2}dt
\\\leq C \int_{0}^{T}\mathbb{E}\int_{t}^{t+\delta}\|P_{n}g(\overline{u}_{n}(s))\|_{L^{2}(I)}^{2}dsdt
\\\leq C\delta \mathbb{E}\int_{0}^{T}\|P_{n}g(\overline{u}_{n}(t))\|_{L^{2}(I)}^{2}dt
\\\leq C\delta \mathbb{E}\int_{0}^{T}\|g(\overline{u}_{n}(t))\|_{L^{2}(I)}^{2}dt
\\\leq C\delta \mathbb{E}\int_{0}^{T}(1+\|\overline{u}_{n}\|_{L^{2}(I)}^{2})dt
\\\leq C\delta.
\end{array}
\end{array}
\end{equation}
It follows from (\ref{47})-(\ref{49}) that
\begin{equation*}
\mathbb{E}\sup\limits_{0\leq |\theta|\leq \delta} \int_{0}^{T}\|\overline{ v}_{n}(t+\theta)-\overline{ v}_{n}(t)\|_{H^{-1}(I)}^{2}dt\leq C\delta.
\end{equation*}
By the regularity theory of elliptic equation
\begin{equation*}
\begin{array}{l}
\left\{
\begin{array}{llll}
\overline{u}_{n}-\varepsilon \overline{u}_{nxx}=\overline{ v}_{n}
\\\overline{u}_{n}(0,t)=0=\overline{u}_{n}(1,t),
\end{array}
\right.
\end{array}
\begin{array}{lll}
{\rm{in}}~I\\
\\
\end{array}
\end{equation*}
we have
\begin{equation*}
\begin{array}{l}
\begin{array}{llll}
\|\overline{u}_{n}(t)\|_{H^{-1}(I)}\leq\|\overline{v}_{n}(t)\|_{H^{-1}(I)},
\end{array}
\end{array}
\end{equation*}
thus, we have (\ref{46}).
\end{proof}
\par \textbf{Step 3. Tightness property of Galerkin solutions.}

We may rewrite Lemma \ref{L7} in the following more convenient form.
\par
By the same way as in \cite[P919]{S2}, according to the priori estimates (\ref{43})(\ref{44})(\ref{45})(\ref{46}), we obtain that
\begin{lemma}
For any $1\leq p<\infty$ and for any sequences $\mu_{m},\nu_{m}$ converging to $0$ such that the series $\sum\limits_{m=1}^{\infty}\frac{\mu_{m}^{\frac{1}{2}}}{\nu_{m}}$ converges, $\{\overline{u}_{n}:n\in \mathbb{N}\}$ is bounded in $X_{p,\mu_{m},\nu_{m}}^{1}$ (the explicit definition of the space $X_{p,\mu_{m},\nu_{m}}^{1}$ can be found in Section 2) for any $m.$
\end{lemma}

\par
Let
$$X= C([0,T];\mathbb{R}^{1})\times L^{2}(0,T;L^{2}(I))$$
and $\mathcal{B}(X)$ be the $\sigma-$algebra of the Borel sets of $X.$
\par
For each $n,$ let $\Phi_{n}$ be the map
\begin{equation*}
\begin{array}{l}
\begin{array}{llll}
\Phi_{n}:~~\overline{\Omega}~~~~\rightarrow ~~~~~~~~~~~~~~~X
\\~~~~~~~\overline{\omega}~~~~\rightarrow~~~~~ (\overline{B}(\overline{\omega}),\overline{u}_{n}(\overline{\omega})),
\end{array}
\end{array}
\end{equation*}
and $\Pi_{n}$ be a probability measure on $(X,\mathcal{B}(X))$ defined by
$$\Pi_{n}(A)=\overline{\mathbb{P}}(\Phi^{-1}_{n}(A)), A\in \mathcal{B}(X).$$
\begin{proposition}\label{Pro4}
The family of probability measures $\{\Pi_{n}:n=1,2,3,...\}$ is tight in $X.$
\end{proposition}
\begin{proof}
For any $\rho >0,$ we should find the compact subsets
\begin{equation*}
\begin{array}{l}
\begin{array}{llll}
 \Sigma_{\rho}\subset C([0,T];\mathbb{R}^{1}), Y_{\rho}\subset L^{2}(0,T;L^{2}(I)),
\end{array}
\end{array}
\end{equation*}
such that
\begin{eqnarray}
\label{41}\overline{\mathbb{P}}(\overline{\omega}:\overline{B}(\overline{\omega},\cdot)\notin\Sigma_{\rho})\leq\frac{\rho}{2},
\\\label{50}\overline{\mathbb{P}}(\overline{\omega}:\overline{u}_{n}(\overline{\omega},\cdot)\notin Y_{\rho})\leq\frac{\rho}{2}
.
\end{eqnarray}

Noting the formula
\begin{equation*}
\overline{\mathbb{E}}|\overline{B}(t_{2})-\overline{B}(t_{1})|^{2i}=(2i-1)!(t_{2}-t_{1})^{i},i=1,2,...
\end{equation*}
we define
\begin{equation*}
\Sigma_{\rho}\triangleq\left\{B(\cdot)\in C([0,T];\mathbb{R}^{1}):\sup\limits_{t_{1},t_{2}\in [0,T], |t_{2}-t_{1}|\leq \frac{1}{n^{6}}}n|B(t_{2})-B(t_{1})|\leq L_{\rho}\right\}
\end{equation*}
where $n\in \mathbb{N},$ $L_{\rho}$ is a constant depending on $\rho$ and will be chosen later.
\par
By the Chebyshev inequality, we get
\begin{equation*}
\begin{array}{l}
\begin{array}{llll}
~~~\displaystyle\overline{\mathbb{P}}(\overline{\omega}:\overline{B}(\overline{\omega},\cdot)\notin\Sigma_{\rho})
\\ \displaystyle\leq\overline{\mathbb{P}}\left(\bigcup_{n}\left\{\omega:\sup\limits_{t_{1},t_{2}\in [0,T], |t_{2}-t_{1}|\leq \frac{1}{n^{6}}}|\overline{B}(t_{2})-\overline{B}(t_{1})|> \frac{L_{\rho}}{n}\right\}\right)
\\ \displaystyle \leq\sum\limits_{n=1}^{\infty}\sum\limits_{i=0}^{n^{6}-1}(\frac{n}{L_{\rho}})^{4}\overline{\mathbb{E}}\sup\limits_{\frac{iT}{n^{6}} \leq t\leq\frac{(i+1)T}{n^{6}}}|\overline{B}(t)-\overline{B}(\frac{iT}{n^{6}})|^{4}
\\ \displaystyle \leq C\sum\limits_{n=1}^{\infty}(\frac{n}{L_{\rho}})^{4}(Tn^{-6})^{2}n^{6}
\\ \displaystyle=\frac{C}{L_{\rho}^{4}}\sum\limits_{n=1}^{\infty}\frac{1}{n^{2}},
\end{array}
\end{array}
\end{equation*}
we choose $L_{\rho}^{4}=2C\rho^{-1}\sum\limits_{n=1}^{\infty}\frac{1}{n^{2}}$ to get (\ref{41}).
\par
Let $Y_{\rho}^{1}$ be a ball of radius $M_{\rho}$ in $Y_{\mu_{m},\nu_{m}}^{1}$ (the explicit definition of the space $Y_{\mu_{m},\nu_{m}}^{1}$ can be found in Section 2), centered at zero, namely $Y_{\rho}^{1}=\{u\in Y_{\mu_{m},\nu_{m}}^{1}~|~\|u\|_{Y_{\mu_{m},\nu_{m}}^{1}}\leq M_{\rho}\}.$ From Corollary \ref{c1}, $Y_{\rho}^{1}$ is a compact subset of $L^{2}(0,T;L^{2}(I)),$ and
\begin{equation*}
\overline{\mathbb{P}}(\overline{\omega}:\overline{u}_{n}(\overline{\omega},\cdot)\notin Y_{\rho}^{1})\leq \overline{\mathbb{P}}(\overline{\omega}:\|\overline{u}_{n}\|_{Y_{\mu_{m},\nu_{m}}^{1}}>M_{\rho})
\leq \frac{1}{M_{\rho}}\overline{\mathbb{E}}\|\overline{u}_{n}\|_{Y_{\mu_{m},\nu_{m}}^{1}}\leq \frac{C}{M_{\rho}},
\end{equation*}
choosing $M_{\rho}=2C\rho^{-1},$ we get (\ref{50}).
\par
It follows from (\ref{41}) and (\ref{50}) that
$$\Pi_{n}(\Sigma_{\rho}\times Y_{\rho}^{1})\geq 1-\rho, $$
for any $n\geq1.$
\par
Thus, the family of probability measures $\{\Pi_{n}:n=1,2,3,...\}$ is tight in $X.$
\end{proof}

\par \textbf{Step 4. Applications of Prokhorov Theorem and Skorokhod Theorem.}
\par
By Lemma \ref{L4}, we can find a probability measure $\Pi$ and extract a subsequence from $\Pi_{n}$ such that
\begin{equation*}
\begin{array}{l}
\begin{array}{llll}
\Pi_{n_{i}}\rightarrow\Pi
\end{array}
\end{array}
\end{equation*}
weakly in $X.$
\par
By Lemma \ref{L5}, there exists a probability space $(\Omega,\mathcal{F},\mathbb{P})$ and random variables $(u_{n_{i}},B_{n_{i}}),$  $(u,B)$ on
$(\Omega,\mathcal{F},\mathbb{P})$ with values in $X$ such that the probability law of $(u_{n_{i}},B_{n_{i}})$ is $\Pi_{n_{i}}.$ Furthermore,
$$(u_{n_{i}},B_{n_{i}})\rightarrow (u,B)~~{\rm{in}}~~X~~P-a.s.$$
and the probability law of $(u,B)$ is $\Pi.$
\par
Set $$\mathcal{F}_{t}=\sigma\{u(s),B(s)\}_{s\in [0,t]}.$$
By the idea in \cite{S2,S3}, we can know $B(t)$ is a $\mathcal{F}_{t}-$standard Wiener process.
\par
We claim that $(u_{n_{i}},B_{n_{i}})$ verifies the following
$dt\otimes d\mathbb{P}-$almost everywhere:
\begin{equation}\label{51}
\begin{array}{l}
\begin{array}{llll}
[(u_{n_{i}}(t),\varphi)+\varepsilon(u_{n_{i}x}(t),\varphi_{x})]-[(u_{n_{i}0},\varphi)+\varepsilon(u_{n_{i}0x},\varphi_{x})]+\int_{0}^{t}((u_{n_{i}x},\varphi_{x})+(P_{n_{i}}u_{n_{i}}^{3}-u_{n_{i}},\varphi))ds
\\~~~~~~~~~~~~~~~~~=\int_{0}^{t}(g(u_{n_{i}}),\varphi)dB_{n_{i}}
\end{array}
\end{array}
\end{equation}
for all $\varphi\in H^{1}_{0}(I).$
\par
Indeed, we set
\begin{equation*}
\begin{array}{l}
\begin{array}{llll}
\xi_{n}(t)=[\overline{u}_{n}(t)-\varepsilon \overline{u}_{nxx}(t)]-[u_{n0}-\varepsilon u_{n0xx}]+\int_{0}^{t}(-\overline{u}_{nxx}+P_{n}\overline{u}_{n}^{3}-\overline{u}_{n})ds
\\~~~~~~~~~~-\int_{0}^{t}P_{n}g(\overline{u}_{n})d\overline{B},
\\ \eta_{n_{i}}(t)=[u_{n_{i}}(t)-\varepsilon u_{n_{i}xx}(t)]-[u_{n_{i}0}-\varepsilon u_{n_{i}0xx}]+\int_{0}^{t}(-u_{n_{i}xx}+P_{n_{i}}u_{n_{i}}^{3}-u_{n_{i}})ds
\\~~~~~~~~~~-\int_{0}^{t}P_{n_{i}}g(u_{n_{i}})dB_{n_{i}},
\\ X_{n}=\int_{0}^{T}\|\xi_{n}(t)\|_{H^{-1}(I)}^{2}dt,
\\
Y_{n_{i}}=\int_{0}^{T}\|\eta_{n_{i}}(t)\|_{H^{-1}(I)}^{2}dt.
\end{array}
\end{array}
\end{equation*}
It is easy to see almost surely $X_{n}=0,$
hence, in particular,
$\overline{\mathbb{E}}\frac{X_{n}}{1+X_{n}}=0.$
\par
Next, we show that
\begin{equation*}
\mathbb{E}\frac{Y_{n_{i}}}{1+Y_{n_{i}}}=0,
\end{equation*}
which will imply (\ref{51}).
\par
Indeed, motivated by \cite{S2}, we introduce a regularization of $g,$ given by
\begin{equation*}
g^{\rho}(y(t))=\frac{1}{\rho}\int_{0}^{t}\beta\left(-\frac{t-s}{\rho}\right)g(y(s))ds,
\end{equation*}
where $\beta$ is a mollifier. It is easy to check that
\begin{equation*}
\mathbb{E}\int_{0}^{T}\|g^{\rho}(y(t))\|_{L^{2}(I)}^{2}dt\leq \mathbb{E}\int_{0}^{T}\|g(y(t))\|_{L^{2}(I)}^{2}dt
\end{equation*}
and
\begin{equation*}
g^{\rho}(y(\cdot))\rightarrow g(y(\cdot))~~ {\rm{in}}~~L^{2}(\Omega,L^{2}(0,T;L^{2}(I))).
\end{equation*}
Then we denote by $X_{n,\rho}$ and $Y_{n_{i},\rho}$ the analog of $X_{n}$ and $Y_{n_{i}}$ with $g$ replaced by $g^{\rho}.$ Introduce
the mapping
\begin{equation*}
\Phi_{n,\rho}(\overline{B},\overline{u}_{n})=\frac{X_{n,\rho}}{1+X_{n,\rho}},
\end{equation*}
owing to the definition of $X_{n,\rho},$ it is easy to see that $\Phi_{n,\rho}$ is bounded and continuous on $C([0,T],\mathbb{R}^{1})\times L^{2}(0,T;L^{2}(I)).$
Similarly, set
\begin{equation*}
\Psi_{n_{i},\rho}(B_{n_{i}},u_{n_{i}})=\frac{Y_{n_{i},\rho}}{1+Y_{n_{i},\rho}}.
\end{equation*}
According to Lemma \ref{L5}, we have
\begin{equation*}
\mathbb{E}\frac{Y_{n_{i},\rho}}{1+Y_{n_{i},\rho}}=\mathbb{E}\Psi_{n_{i},\rho}(B_{n_{i}},u_{n_{i}})=\int_{S}\Psi_{n_{i},\rho}d\Pi_{n_{i}}
=\overline{\mathbb{E}}\Phi_{n_{i},\rho}(\overline{B},\overline{u}_{n_{i}})=\overline{\mathbb{E}}\frac{X_{n_{i},\rho}}{1+X_{n_{i},\rho}},
\end{equation*}
therefore,
\begin{equation*}
\begin{array}{l}
\begin{array}{llll}
\displaystyle ~~~\mathbb{E}\frac{Y_{n_{i}}}{1+Y_{n_{i}}}-\overline{\mathbb{E}}\frac{X_{n_{i}}}{1+X_{n_{i}}}
\\ \displaystyle=\mathbb{E}\left(\frac{Y_{n_{i}}}{1+Y_{n_{i}}}-\frac{Y_{n_{i},\rho}}{1+Y_{n_{i},\rho}}\right)
+\mathbb{E}\frac{Y_{n_{i},\rho}}{1+Y_{n_{i},\rho}}-\overline{\mathbb{E}}\frac{X_{n_{i},\rho}}{1+X_{n_{i},\rho}}
+\overline{\mathbb{E}}\left(\frac{X_{n_{i},\rho}}{1+X_{n_{i},\rho}}-\frac{X_{n_{i}}}{1+X_{n_{i}}}\right)
\\ \displaystyle=\mathbb{E}\left(\frac{Y_{n_{i}}}{1+Y_{n_{i}}}-\frac{Y_{n_{i},\rho}}{1+Y_{n_{i},\rho}}\right)
+\overline{\mathbb{E}}\left(\frac{X_{n_{i},\rho}}{1+X_{n_{i},\rho}}-\frac{X_{n_{i}}}{1+X_{n_{i}}}\right).
\end{array}
\end{array}
\end{equation*}
It is clear that
\begin{equation*}
\begin{array}{l}
\begin{array}{llll}
~~~\displaystyle\left||\mathbb{E}\frac{Y_{n_{i}}}{1+Y_{n_{i}}}|-|\overline{\mathbb{E}}\frac{X_{n_{i}}}{1+X_{n_{i}}}|\right|
\\ \displaystyle\leq \left|\mathbb{E}\frac{Y_{n_{i}}}{1+Y_{n_{i}}}-\overline{\mathbb{E}}\frac{X_{n_{i}}}{1+X_{n_{i}}}\right|
\\ \displaystyle\leq \left|\mathbb{E}\left(\frac{Y_{n_{i}}}{1+Y_{n_{i}}}-\frac{Y_{n_{i},\rho}}{1+Y_{n_{i},\rho}}\right)\right|
+\left|\overline{\mathbb{E}}\left(\frac{X_{n_{i},\rho}}{1+X_{n_{i},\rho}}-\frac{X_{n_{i}}}{1+X_{n_{i}}}\right)\right|
\\ \displaystyle \leq C\left(\mathbb{E}\int_{0}^{T}\|g^{\rho}(\overline{u}_{n_{i}}(t))
-g(\overline{u}_{n_{i}}(t))\|_{L^{2}(I)}^{2}dt\right)^{\frac{1}{2}}.
\end{array}
\end{array}
\end{equation*}
As $\rho\rightarrow0,$ it follows that
\begin{equation*}
\left|\mathbb{E}\frac{Y_{n_{i}}}{1+Y_{n_{i}}}\right|=\left|\overline{\mathbb{E}}\frac{X_{n_{i}}}{1+X_{n_{i}}}\right|=0.
\end{equation*}
It follows that (\ref{51}) holds.
\par \textbf{Step 5. Passage to the limit.}
\par
From (\ref{51}), it follows that $u_{n_{i}}$ satisfies the results of (\ref{43})(\ref{44})(\ref{45})(\ref{46}), we can extract from $u_{n_{i}}$ a subsequence still denoted with the same fashion and a function $u$ such that
\begin{equation*}
\begin{array}{l}
\begin{array}{llll}
u_{n_{i}}\rightarrow u~~{\rm{weakly}}~\ast~{\rm{in}}~~L^{p}(\Omega,L^{\infty}(0,T;L^{2}(I))),
\\u_{n_{i}}\rightarrow u~~{\rm{weakly~in}}~~L^{p}(\Omega,L^{2}(0,T;H^{1}(I))),
\\u_{n_{i}}\rightarrow u~~{\rm{weakly~in}}~~L^{4}(\Omega,L^{4}(0,T;L^{4}(I))),
\\u_{n_{i}}\rightarrow u~~{\rm{strongly~in}}~~L^{2}(0,T;L^{2}(I))~~P-a.s.
\end{array}
\end{array}
\end{equation*}
By Vitali's convergence theorem, we have
\begin{equation*}
\begin{array}{l}
\begin{array}{llll}
u_{n_{i}}\rightarrow u~~{\rm{strongly~in}}~~L^{2}(\Omega,L^{2}(0,T;L^{2}(I))).
\end{array}
\end{array}
\end{equation*}
It follows from these facts that we can extract again from $u_{n_{i}}$ a subsequence still denoted by the same symbols such that
\begin{eqnarray}
\label{52}&&u_{n_{i}}\rightarrow u~~{\rm{almost ~~everywhere~~}}dt\otimes d\mathbb{P}-~{\rm{in}}~~L^{2}(I),
\\\label{53}&&u_{n_{i}}\rightarrow u~~{\rm{almost ~~everywhere~~}}dt\otimes dx\otimes d\mathbb{P}~{\rm{in}}~~[0,T]\times I\times \Omega.
\end{eqnarray}
It follows from (\ref{53}) that for any $t\in [0,T],$
\begin{equation}
\label{56}u_{n_{i}}\rightarrow u~~{\rm{almost ~~everywhere~~}}dt\otimes dx\otimes d\mathbb{P}~{\rm{in}}~~[0,t]\times I\times \Omega.
\end{equation}

\par
Since $u_{n_{i}}$ is bounded in $L^{4}(\Omega,L^{4}(0,T;L^{4}(I)))$, we have
$u^{3}_{n_{i}}$ is bounded in $L^{\frac{4}{3}}([0,T]\times I\times \Omega),$
Combining this and (\ref{56}), we deduce that
\begin{equation}
\label{}u_{n_{i}}^{3}\rightarrow u^{3}~~{\rm{weakly~in}}~~L^{\frac{4}{3}}([0,T]\times I\times \Omega).
\end{equation}
\par
By (\ref{52}), the continuity of $g,$ and the applicability of Vitali's convergence theorem we have
\begin{equation}
\displaystyle P_{n_{i}}g(u_{n_{i}})\rightarrow g(u)~~{\rm{strongly~in}}~~L^{2}(\Omega,L^{2}(0,T;L^{2}(I))).
\end{equation}
\par
By the idea in \cite[P284]{B2} and \cite[P922]{S2}, we can know
\begin{equation}
\label{54}\int_{0}^{t}P_{n_{i}}g(u_{n_{i}})dB_{n_{i}}\rightarrow \int_{0}^{t}g(u)dB~~{\rm{weakly~in}}~L^{2}(\Omega,L^{2}(I))
\end{equation}
for any $t\in [0,T]$.
\par
As $$u_{n_{i}}\rightarrow u~~{\rm{weakly~in}}~~L^{p}(\Omega,L^{2}(0,T;H^{1}(I))),$$
then
\begin{equation}\label{55}
u_{n_{i}xx}\rightarrow u_{xx}~~{\rm{weakly~in}}~~L^{2}(\Omega,L^{2}(0,T;H^{-1}(I))).
\end{equation}
\par
Collecting all the convergence results (\ref{52})-(\ref{55}), we deduce that
$(u,B)$ verifies the following equation $dt\otimes d\mathbb{P}-$almost everywhere:
\begin{equation*}
\begin{array}{l}
\begin{array}{llll}
[(u(t),\varphi)+\varepsilon(u_{x}(t),\varphi_{x})]-[(u_{0},\varphi)+\varepsilon(u_{0x},\varphi_{x})]+\int_{0}^{t}((u_{x},\varphi_{x})+(u^{3}-u,\varphi))ds
\\~~~~~~~~~~~~~~~~~=\int_{0}^{t}(g(u),\varphi)dB
\end{array}
\end{array}
\end{equation*}
for all $\varphi \in H^{1}_{0}(I).$
\par
Estimates (\ref{33})-(\ref{35}) follow from passing to the limits in (\ref{44}), (\ref{45}) and (\ref{46}).

\section{Proof of Theorem \ref{Th4} }
This section is motivated by \cite{S4}.
\par
It follows from Theorem \ref{Th3} that there exists a
sequence of weak martingale solutions
$$\{(\Omega^{\varepsilon},\mathcal{F}^{\varepsilon},\mathbb{P}^{\varepsilon}),(\mathcal{F}^{\varepsilon}_{t})_{0\leq t\leq T},u^{\varepsilon},B^{\varepsilon}\}$$
satisfy the inequalities
\begin{equation}\label{59}
\begin{array}{l}
\begin{array}{llll}
\mathbb{E} \sup\limits_{0\leq t\leq T}(\|u^{\varepsilon}(t)\|_{L^{2}(I)}^{2}+\varepsilon\|u_{x}^{\varepsilon}(t)\|_{L^{2}(I)}^{2})^{\frac{p}{2}}\leq C(p,T),
\\\mathbb{E}\left(\int_{0}^{T}(\|u_{x}^{\varepsilon}(t)\|_{L^{2}(I)}^{2}+\|u^{\varepsilon}\|_{L^{4}(I)}^{4})dt\right)^{\frac{p}{2}}\leq C(p,T),
\\\mathbb{E} \sup\limits_{0\leq |\theta|\leq \delta\leq 1} \int_{0}^{T}\|u^{\varepsilon}(t+\theta)-u^{\varepsilon}(t)\|_{H^{-1}(I)}^{2}dt\leq C(p,T)\delta,
\end{array}
\end{array}
\end{equation}
where $C(p,T)$ is a constant independent of $\varepsilon.$
\par
By the same way as in \cite[P919]{S2} and \cite[P2237]{S4}, according to the priori estimates (\ref{59}), we obtain that
\begin{lemma}
For any $1\leq p<\infty$ and for any sequences $\mu_{m},\nu_{m}$ converging to $0$ such that the series $\sum\limits_{m=1}^{\infty}\frac{\mu_{m}^{\frac{1}{2}}}{\nu_{m}}$ converges, $\{u^{\varepsilon}\}_{0<\varepsilon<1}$ is bounded in $X_{p,\mu_{m},\nu_{m}}^{1}$ (the explicit definition of the space $X_{p,\mu_{m},\nu_{m}}^{1}$ can be found in Section 2) for any $m.$
\end{lemma}
\par
Let
$$X= C([0,T];\mathbb{R}^{1}) \times L^{2}(0,T;L^{2}(I))$$
and $\mathcal{B}(X)$ be the $\sigma-$algebra of the Borel sets of $X.$
\par
For each $\varepsilon,$ let $\Phi_{\varepsilon}$ be the map
\begin{equation*}
\begin{array}{l}
\begin{array}{llll}
\Phi_{\varepsilon}:~~\Omega^{\varepsilon}~~~~\rightarrow ~~~~~~~~~~~~~~~X
\\~~~~~~~\omega~~~~\rightarrow~~~~~ (B^{\varepsilon}(\omega),u^{\varepsilon}(\omega)),
\end{array}
\end{array}
\end{equation*}
and $\Pi_{\varepsilon}$ be a probability measure on $(X,\mathcal{B}(X))$ defined by
$$\Pi_{\varepsilon}(A)=\mathbb{P}^{\varepsilon}(\Phi^{-1}_{\varepsilon}(A)), A\in \mathcal{B}(X).$$
\begin{proposition}
The family of probability measures $\{\Pi_{\varepsilon}:\varepsilon\in [0,1]\}$ is tight in $X.$
\end{proposition}
\begin{proof}
\par We use the same method as in Proposition \ref{Pro4}.
\par
For any $\rho >0,$ we should find the compact subsets
\begin{equation*}
\begin{array}{l}
\begin{array}{llll}
 \Sigma_{\rho}\subset C([0,T];\mathbb{R}^{1}), Y_{\rho}^{1}\subset L^{2}(0,T;L^{2}(I)),
\end{array}
\end{array}
\end{equation*}
such that
\begin{eqnarray}
\label{57}\mathbb{P}^{\varepsilon}(\omega:B^{\varepsilon}(\omega,\cdot)\notin\Sigma_{\rho})\leq\frac{\rho}{2},
\\\label{58}\mathbb{P}^{\varepsilon}(\omega:u^{\varepsilon}(\omega,\cdot)\notin Y_{\rho}^{1})\leq\frac{\rho}{2}
.
\end{eqnarray}

Noting the formula
\begin{equation*}
\mathbb{E}^{\varepsilon}|B^{\varepsilon}(t_{2})-B^{\varepsilon}(t_{1})|^{2i}=(2i-1)!(t_{2}-t_{1})^{i},i=1,2,...
\end{equation*}
we define
\begin{equation*}
\begin{array}{l}
\begin{array}{llll}
\Sigma_{\rho}\triangleq\left\{B(\cdot)\in C([0,T];\mathbb{R}^{1}):\sup\limits_{t_{1},t_{2}\in [0,T], |t_{2}-t_{1}|\leq \frac{1}{n^{6}}}n|B(t_{2})-B(t_{1})|\leq L_{\rho}\right\},\\
Y_{\rho}^{1}=\left\{u\in Y_{\mu_{m},\nu_{m}}^{1}~|~\|u\|_{Y_{\mu_{m},\nu_{m}}^{1}}\leq M_{\rho}\right\}.
\end{array}
\end{array}
\end{equation*}
where $n\in \mathbb{N},$ $L_{\rho},M_{\rho}$ two constants depending on $\rho$ and will be chosen later.
\par
By the Chebyshev inequality and the same argument as in Proposition \ref{Pro4}, we get
\begin{equation*}
\begin{array}{l}
\begin{array}{llll}
\mathbb{P}^{\varepsilon}(\omega:B^{\varepsilon}(\omega,\cdot)\notin\Sigma_{\rho})
\leq\frac{C}{L_{\rho}^{4}}\sum\limits_{n=1}^{\infty}\frac{1}{n^{2}},
\\
\mathbb{P}^{\varepsilon}(\omega:u^{\varepsilon}(\omega,\cdot)\notin Y_{\rho}^{1})\leq \frac{C}{M_{\rho}},
\end{array}
\end{array}
\end{equation*}
we choose $L_{\rho}^{4}=2C\rho^{-1}\sum\limits_{n=1}^{\infty}\frac{1}{n^{2}}, M_{\rho}=2C\rho^{-1},$ to get (\ref{57}) and (\ref{58}).
\par
It follows from (\ref{57}) and (\ref{58}) that
$$\Pi_{\varepsilon}(\Sigma_{\rho}\times Y_{\rho}^{1})\geq 1-\rho, $$
for any $\varepsilon\in[0,1].$
\par
Thus, the family of probability measures $\{\Pi_{\varepsilon}:\varepsilon\in[0,1]\}$ is tight in $X.$
\end{proof}

\par
From the tightness of $\{\Pi_{\varepsilon}:\varepsilon\in[0,1]\}$ in the Polish space $X$ and Prokhorov¡¯s theorem, we
infer the existence of a subsequence $\Pi_{\varepsilon_{i}}$ of probability measures and a probability measure $\Pi$
such that $\Pi_{\varepsilon_{i}}\rightharpoonup\Pi$ weakly as $i\rightarrow\infty.$
\par
By Lemma \ref{L5}, there exists a probability space $(\Omega,\mathcal{F},\mathbb{P})$ and random variables $(\tilde{u}^{\varepsilon_{i}},\tilde{B}^{\varepsilon_{i}}),$  $(u,B)$ on $(\Omega,\mathcal{F},\mathbb{P})$ with values in $X$ such that
\begin{equation*}
\begin{array}{l}
\begin{array}{llll}
\mathcal{L}(\tilde{u}^{\varepsilon_{i}},\tilde{B}^{\varepsilon_{i}})=\Pi_{\varepsilon_{i}},~~\mathcal{L}(u,B)=\Pi,
\\(\tilde{u}^{\varepsilon_{i}},\tilde{B}^{\varepsilon_{i}})\rightarrow (u,B)~~{\rm{in}}~~X~~P-a.s.
\end{array}
\end{array}
\end{equation*}
By the same argument as in (\ref{51}), we have
\begin{equation}\label{60}
\begin{array}{l}
\begin{array}{llll}
[(\tilde{u}^{\varepsilon_{i}}(t),\varphi)+\varepsilon(\tilde{u}^{\varepsilon_{i}}_{x}(t),\varphi_{x})]-[(\tilde{u}^{\varepsilon_{i}}_{0},\varphi)
+\varepsilon(\tilde{u}^{\varepsilon_{i}}_{0x},\varphi_{x})]+\int_{0}^{t}((\tilde{u}^{\varepsilon_{i}}_{x},\varphi_{x})+(\tilde{u}^{\varepsilon_{i}3}-\tilde{u}^{\varepsilon_{i}},\varphi))ds
\\~~~~~~~~~~~~~~~~~=\int_{0}^{t}(g(\tilde{u}^{\varepsilon_{i}}),\varphi)d\tilde{B}^{\varepsilon_{i}}
\end{array}
\end{array}
\end{equation}
for all $\varphi \in H^{1}_{0}(I).$
\par
From (\ref{60}), it follows that $\tilde{u}^{\varepsilon_{i}}$ satisfies the results of (\ref{43})(\ref{44})(\ref{45})(\ref{46}), we can extract from $\tilde{u}^{\varepsilon_{i}}$ a subsequence still denoted with the same fashion and a function $u$ such that
\begin{equation*}
\begin{array}{l}
\begin{array}{llll}
\tilde{u}^{\varepsilon_{i}}\rightarrow u~~{\rm{weakly}}~\ast~{\rm{in}}~~L^{p}(\Omega,L^{\infty}(0,T;L^{2}(I))),
\\\tilde{u}^{\varepsilon_{i}}\rightarrow u~~{\rm{weakly~in}}~~L^{p}(\Omega,L^{2}(0,T;H^{1}(I))),
\\\tilde{u}^{\varepsilon_{i}}\rightarrow u~~{\rm{weakly~in}}~~L^{4}(\Omega,L^{4}(0,T;L^{4}(I))),
\\\tilde{u}^{\varepsilon_{i}}\rightarrow u~~{\rm{strongly~in}}~~L^{2}(0,T;L^{2}(I))~~P-a.s.
\end{array}
\end{array}
\end{equation*}
By Vitali's convergence theorem, we have
\begin{equation*}
\begin{array}{l}
\begin{array}{llll}
\lim\limits_{i\rightarrow\infty}\mathbb{E}\|\tilde{u}^{\varepsilon_{i}}-u\|_{L^{2}(0,T;L^{2}(I))}^{2}=0,
\end{array}
\end{array}
\end{equation*}
according to this equality, Theorem \ref{Th1}, \cite[P284]{B2}, \cite[P1126,Lemma 2.1]{D1} and \cite[P151,Lemma 3.1]{G1}, it is easy to see that for any $\delta>0,$ we have
\begin{equation*}
\begin{array}{l}
\begin{array}{llll}
\lim\limits_{i\rightarrow\infty}\mathbb{P}(\|(\tilde{u}^{\varepsilon_{i}}(t),\varphi)-(u(t),\varphi)\|_{L^{2}(0,T)}>\delta)=0,
\\
\lim\limits_{i\rightarrow\infty}\mathbb{P}(\|\int_{0}^{t}( \tilde{u}^{\varepsilon_{i}}_{x}(s),\varphi_{x})ds-\int_{0}^{t}( u_{x}(s),\varphi_{x})ds\|_{L^{2}(0,T)}>\delta)=0,
\\
\lim\limits_{i\rightarrow\infty}\mathbb{P}(\|\int_{0}^{t}(\tilde{u}^{\varepsilon_{i}3}-\tilde{u}^{\varepsilon_{i}},\varphi)ds-\int_{0}^{t}(u^{3}-u,\varphi)ds\|_{L^{2}(0,T)}>\delta)=0,
\\
\lim\limits_{i\rightarrow\infty}\mathbb{P}(\|\int_{0}^{t}(g(\tilde{u}^{\varepsilon_{i}}),\varphi)d\tilde{B}^{\varepsilon_{i}}(s)-\int_{0}^{t}(g(u),\varphi)dB(s)\|_{L^{2}(0,T)}>\delta)=0.
\end{array}
\end{array}
\end{equation*}
\par
It follows from
\begin{equation}
\begin{array}{l}
\begin{array}{llll}
\mathbb{E}\sup\limits_{0\leq t\leq T}|\varepsilon_{i}(\tilde{u}^{\varepsilon_{i}}_{x}(t),\varphi_{x})|^{2}
\\\leq \mathbb{E}\sup\limits_{0\leq t\leq T}\varepsilon_{i}^{2}\|\tilde{u}^{\varepsilon_{i}}_{x}(t)\|_{L^{2}(I)}^{2}\|\varphi_{x}\|_{L^{2}(I)}^{2}
\\\leq \varepsilon_{i}\|\varphi_{x}\|_{L^{2}(I)}^{2}\mathbb{E}\sup\limits_{0\leq t\leq T}\varepsilon_{i}\|\tilde{u}^{\varepsilon_{i}}_{x}(t)\|_{L^{2}(I)}^{2}
\end{array}
\end{array}
\end{equation}
that
\begin{equation}
\begin{array}{l}
\begin{array}{llll}
\lim\limits_{i\rightarrow\infty}\mathbb{E}\sup\limits_{0\leq t\leq T}|\varepsilon_{i}(\tilde{u}^{\varepsilon_{i}}_{x}(t),\varphi_{x})|^{2}=0.
\end{array}
\end{array}
\end{equation}
\par
By taking the limit in probability as $i$ goes to infinity in (\ref{60}), we deduce that
$(u,B)$ verifies the following equation $dt\otimes d\mathbb{P}-$almost everywhere:
\begin{equation}
\begin{array}{l}
\begin{array}{llll}
(u(t),\varphi)-(u_{0},\varphi)+\int_{0}^{t}((u_{x},\varphi_{x})+(u^{3}-u,\varphi))ds=\int_{0}^{t}(g(u),\varphi)dB
\end{array}
\end{array}
\end{equation}
for all $\varphi \in H^{1}_{0}(I).$
Namely, $\{(\Omega,\mathcal{F},\mathbb{P}),(\mathcal{F}_{t})_{0\leq t\leq T},u,B\}$ is a weak martingale solution of problem (\ref{61}).

\section{Proof of Theorem \ref{Th1}}
\par
If there is no danger of confusion, we shall omit the subscript $\varepsilon,$ we use $u$ instead of $u^{\varepsilon}$ and $v$ instead of $v^{\varepsilon}.$
\par
The proof is divided into several steps.
\subsection{ Local existence}.
\par
Based on Proposition \ref{Pro1}, we can obtain the following result.
\begin{proposition}\label{Pro2}
For any $\varepsilon\in [0,\frac{1}{2}], T>0.$ If
\begin{equation*}
\begin{array}{l}
\begin{array}{llll}
u_{0}\in L^{2}(\Omega;H^{2}(I)\cap H^{1}_{0}(I)),\\
\|f(u_{1})-f(u_{2})\|_{L^{2}(I)}\leq L\|u_{1}-u_{2}\|_{H^{1}(I)},\\
\|f(u)\|_{L^{2}(I)}\leq L(1+\|u\|_{H^{1}(I)}),
\end{array}
\end{array}
\end{equation*}
then equation
\begin{equation}\label{9}
\begin{array}{l}
\left\{
\begin{array}{llll}
d(u^{\varepsilon}-\varepsilon u^{\varepsilon}_{xx})+(-u^{\varepsilon}_{xx}+f(u^{\varepsilon}))dt=g(u^{\varepsilon})dB
\\u^{\varepsilon}(0,t)=0=u^{\varepsilon}(1,t)
\\u^{\varepsilon}(0)=u_{0}

\end{array}
\right.
\end{array}
\begin{array}{lll}
{\rm{in}}~I\times(0,T)\\
{\rm{in}}~(0,T)\\
{\rm{in}}~I,
\end{array}
\end{equation}
has a unique solution $u^{\varepsilon}\in L^{2}(\Omega;C([0,T];H^{2}(I)\cap H^{1}_{0}(I)))$
and
\begin{equation}\label{10}
\begin{array}{l}
\begin{array}{llll}
\mathbb{E} \sup\limits_{0\leq t\leq T}(\|u^{\varepsilon}_{x}(t)\|_{L^{2}(I)}^{2}+\varepsilon\|u^{\varepsilon}_{xx}(t)\|_{L^{2}(I)}^{2})
+\mathbb{E} \int_{0}^{T}\|u^{\varepsilon}_{xx}(t)\|_{L^{2}(I)}^{2}dt
\\\leq C\mathbb{E}(\| u_{0x}\|_{L^{2}(I)}^{2}+\|u_{0xx}\|_{L^{2}(I)}^{2}),
\end{array}
\end{array}
\end{equation}
where $C=C(L,T,I).$
\end{proposition}
\begin{proof}
The main idea in this part comes from \cite{K1}.
\par
We set
\begin{equation*}
\begin{array}{l}
\begin{array}{llll}
u_{0}(t)=u_{0},
\end{array}
\end{array}
\end{equation*}
$u_{n+1}(t)$ is the solution of
\begin{equation}
\begin{array}{l}
\left\{
\begin{array}{llll}
d(u-\varepsilon u_{xx})+(-u_{xx}+f(u_{n}(t)))dt=g(u_{n}(t))dB
\\u(0,t)=0=u(1,t)
\\u(0)=u_{0}

\end{array}
\right.
\end{array}
\begin{array}{lll}
{\rm{in}}~I\times(0,T)\\
{\rm{in}}~(0,T)\\
{\rm{in}}~I.
\end{array}
\end{equation}
Then,
\begin{equation}
\begin{array}{l}
\left\{
\begin{array}{llll}
d(u_{n+1}-u_{n}-\varepsilon (u_{n+1}-u_{n})_{xx})
\\~~~~~~+(-(u_{n+1}-u_{n})_{xx}+f(u_{n}(t))-f(u_{n-1}(t)))dt=(g(u_{n}(t))-g(u_{n-1}(t)))dB
\\(u_{n+1}-u_{n})(0,t)=0=(u_{n+1}-u_{n})(1,t)
\\(u_{n+1}-u_{n})(0)=0

\end{array}
\right.
\end{array}
\begin{array}{lll}
{\rm{in}}~I\times(0,T)\\
{\rm{in}}~(0,T)\\
{\rm{in}}~I,
\end{array}
\end{equation}
It follows from (\ref{8}) that
\begin{equation}
\begin{array}{l}
\begin{array}{llll}
\mathbb{E} \sup\limits_{0\leq s\leq t}(\|(u_{n+1}-u_{n})_{x}(s)\|_{L^{2}(I)}^{2}+\varepsilon\|(u_{n+1}-u_{n})(s)_{xx}\|_{L^{2}(I)}^{2})
\\+\mathbb{E} \int_{0}^{t}\| (u_{n+1}-u_{n})_{xx}(s)\|_{L^{2}(I)}^{2}ds
\\\leq C[\mathbb{E} \int_{0}^{t}\|f(u_{n}(t))-f(u_{n-1}(s))\|_{L^{2}(I)}^{2}ds+\mathbb{E} \int_{0}^{t}\|(g(u_{n}(s))-g(u_{n-1}(s)))\|_{H^{1}(I)}^{2}ds]
\\\leq C[\mathbb{E} \int_{0}^{t}L^{2}\|u_{n}(t)-u_{n-1}(s)\|_{L^{2}(I)}^{2}ds+\mathbb{E} \int_{0}^{t}L^{2}\|u_{n}(s)-u_{n-1}(s)\|_{H^{1}(I)}^{2}ds]
\\\leq C L^{2}\mathbb{E}\int_{0}^{t}\sup\limits_{0\leq \tau\leq s}\|u_{n}(\tau)-u_{n-1}(\tau)\|_{H^{1}(I)}^{2}ds
\\\leq C L^{2}\mathbb{E}\int_{0}^{t}\sup\limits_{0\leq \tau\leq s}\|(u_{n}-u_{n-1})_{x}(\tau)\|_{L^{2}(I)}^{2}ds
\\\leq C L^{2}\mathbb{E} \int_{0}^{t}\sup\limits_{0\leq \tau\leq s}(\|(u_{n}-u_{n-1})_{x}(\tau)\|_{L^{2}(I)}^{2}+\varepsilon\| (u_{n}-u_{n-1})_{xx}(\tau)\|_{L^{2}(I)}^{2})ds.

\end{array}
\end{array}
\end{equation}
We define
\begin{equation}
\begin{array}{l}
\begin{array}{llll}
Q_{n}(t)=\mathbb{E} \sup\limits_{0\leq s\leq t}(\|(u_{n+1}-u_{n})_{x}(s)\|_{L^{2}(I)}^{2}+\varepsilon\|(u_{n+1}-u_{n})_{xx}(s)\|_{L^{2}(I)}^{2}),

\end{array}
\end{array}
\end{equation}
then, we have
\begin{equation}
\begin{array}{l}
\begin{array}{llll}
Q_{n}(t)\leq C L^{2}\int_{0}^{t}Q_{n-1}(s)ds.

\end{array}
\end{array}
\end{equation}
It is easy to see that
\begin{equation}
\begin{array}{l}
\begin{array}{llll}
Q_{1}(t)\leq C_{0},\\
Q_{n}(t)\leq \frac{C_{0}C^{n} L^{2n}}{n!}t^{n},

\end{array}
\end{array}
\end{equation}
which yields
\begin{equation}
\begin{array}{l}
\begin{array}{llll}
\sum_{n=1}^{+\infty}\sqrt{Q_{n}(T)}<+\infty.
\end{array}
\end{array}
\end{equation}
Consequently, $\{u_{n}\}_{n=1}^{+\infty}$ is a Cauchy sequence in $L^{2}(\Omega,C([0,T];H^{2}(I)))$. Then it is
easy to see that the limit gives a solution of (\ref{9}).
\par
According to Proposition \ref{Pro1} (3), we have
\begin{equation*}
\begin{array}{l}
\begin{array}{llll}
\mathbb{E} \sup\limits_{0\leq s\leq t}(\|u_{x}(t)\|_{L^{2}(I)}^{2}+\varepsilon\|u_{xx}(t)\|_{L^{2}(I)}^{2})
+\mathbb{E} \int_{0}^{t}\|u^{\varepsilon}_{xx}(t)\|_{L^{2}(I)}^{2}ds
\\\leq C[\mathbb{E}(\|u_{0x}\|_{L^{2}(I)}^{2}+\|u_{0xx}\|_{L^{2}(I)}^{2})+\mathbb{E} \int_{0}^{t}\|f(u(s))\|_{L^{2}(I)}^{2}ds+\mathbb{E} \int_{0}^{t}\|g(u(s))\|_{H^{1}(I)}^{2}ds]
\\\leq C(L)[\mathbb{E}(\|u_{0x}\|_{L^{2}(I)}^{2}+\|u_{0xx}\|_{L^{2}(I)}^{2})+\mathbb{E} \int_{0}^{t}(1+\|u(s)\|_{H^{1}(I)}^{2})ds]
\\\leq C(L)[\mathbb{E}(\|u_{0x}\|_{L^{2}(I)}^{2}+\|u_{0xx}\|_{L^{2}(I)}^{2})+\mathbb{E} \int_{0}^{t}(1+\|u_{x}(s)\|_{L^{2}(I)}^{2})ds]
\\\leq C(L)[\mathbb{E}(\|u_{0x}\|_{L^{2}(I)}^{2}+\|u_{0xx}\|_{L^{2}(I)}^{2})+T+ \int_{0}^{t}\mathbb{E}\sup\limits_{0\leq \tau\leq s}\| u_{x}(\tau)\|_{L^{2}(I)}^{2}ds],
\end{array}
\end{array}
\end{equation*}
the Ironwall inequality now implies (\ref{10}).
\par
The uniqueness can also be obtained from the Ironwall inequality.
\end{proof}
\par
Let $\rho\in C^{\infty}_{0}(\mathbb{R})$ be a cut-off function such that $\rho(r)=1$ for $r\in[0,1]$ and $\rho(r)=0$ for $r\geq 2.$ For any $R>0,y\in H^{1}(I)$
and $t\in [0,T],$ we set
\begin{equation*}
\begin{array}{l}
\begin{array}{llll}
\rho_{R}(y)=\rho(\frac{\|y\|_{H^{1}(I)}}{R}),
\\f_{R}(y)=\rho_{R}(y)y^{3}.
\end{array}
\end{array}
\end{equation*}
It is easy to see
\begin{equation*}
\begin{array}{l}
\begin{array}{llll}
\|f_{R}(y_{1})-f_{R}(y_{2})\|_{L^{2}(I)}
\leq CR^{2}\|y_{1}-y_{2}\|_{H^{1}(I)}.
\end{array}
\end{array}
\end{equation*}
The truncated equation corresponding to (\ref{13}) is the
following stochastic partial differential equation:
\begin{equation}\label{14}
\begin{array}{l}
\left\{
\begin{array}{llll}
d(u-\varepsilon u_{xx})+(- u_{xx}+f_{R}(u)-u)dt=g(u)dB
\\u(x,t)=0
\\u(0)=u_{0}
\end{array}
\right.
\end{array}
\end{equation}
\par It follows from  Proposition \ref{Pro2} that (\ref{14})
has a unique solution $u_{R}\in L^{2}(\Omega;C([0,T];H^{2}(I)\cap H^{1}_{0}(I))).$
We define
\begin{equation*}
\tau_{R}=\inf\{t\geq 0~|~\|u_{R}(t)\|_{H^{2}(I)}\geq R\}
\end{equation*}
with the usual convention that $\inf \emptyset=+\infty.$
\par
Since the sequence
of stopping times $\tau_{R}$ is non-decreasing on $R,$ we can put
$$\tau^{*}=\lim\limits_{R\rightarrow \infty}\tau_{R}.$$
We can define a local solution to (\ref{14}) as
$$u(t)=u_{R}(t)$$
on $[0,\tau_{R}],$ which is well defined since
$$u_{R_{1}}(t)=u_{R_{2}}(t)$$
on $[0,\tau_{R_{1}}\wedge \tau_{R_{2}}].$
\par
Indeed, $u_{R_{1}}(t)-u_{R_{2}}(t)$ is the solution of
\begin{equation*}
\begin{array}{l}
\left\{
\begin{array}{llll}
d(h-\varepsilon h_{xx})+(- h_{xx}+f_{R_{1}}(u_{R_{1}})-f_{R_{2}}(u_{R_{2}})-h)dt=[g(u_{R_{1}})-g(u_{R_{2}})]dB
\\h(0,t)=0=h(1,t)
\\h(0)=0,
\end{array}
\right.
\end{array}
\end{equation*}
for $t\leq[0,\tau_{R_{1}}\wedge \tau_{R_{2}}]$ with $R_{1}\leq R_{2},$ it follows from Proposition \ref{Pro1} that
\begin{equation*}
\begin{array}{l}
\begin{array}{llll}
\mathbb{E} \sup\limits_{0\leq s\leq t}(\|h_{x}(s)\|_{L^{2}(I)}^{2}+\varepsilon\|h_{xx}(s)\|_{L^{2}(I)}^{2})
\\\leq C[\mathbb{E} \int_{0}^{t}\|f_{R_{1}}(u_{R_{1}})-f_{R_{2}}(u_{R_{2}})-h\|_{L^{2}(I)}^{2}ds+\mathbb{E} \int_{0}^{t}\|g(u_{R_{1}})-g(u_{R_{2}})\|_{H^{1}(I)}^{2}ds]
\\= C[\mathbb{E} \int_{0}^{t}\|f_{R_{2}}(u_{R_{1}})-f_{R_{2}}(u_{R_{2}})-h\|_{L^{2}(I)}^{2}ds+\mathbb{E} \int_{0}^{t}\|g(u_{R_{1}})-g(u_{R_{2}})\|_{H^{1}(I)}^{2}ds]
\\\leq \beta(t)\mathbb{E} \sup\limits_{0\leq s\leq t}\| h_{x}(s)\|_{L^{2}(I)}^{2},
\end{array}
\end{array}
\end{equation*}
where $\beta(t)$ is a continuous increasing function with
$\beta(0)= 0.$
\par
If we take $t$ sufficiently small, we have $u_{R_{1}}=u_{R_{2}}$
on $[0,t].$ Repeating the same argument in the interval
$[t, 2t]$ and so on yields $$u_{R_{1}}=u_{R_{2}}$$ in the whole interval $[0,\tau_{R_{1}}\wedge \tau_{R_{2}}]$.
\par
At the end, if $\tau^{*}<+\infty$, the definition of $u$ yields
$$\lim\limits_{t\rightarrow \tau^{*}}\|u(t)\|_{H^{2}(I)}=+\infty,$$
which shows that $u$ is a unique local solution to (\ref{14}) on the
interval $[0,\tau^{*}]$, and thus completes the proof.
\subsection{ Global existence}
\par
We will exploit an energy inequality.
\par
For any $T>0,$ set $\tau=\inf\{\tau^{*},T\}$ and $t<\tau.$
\par
\textbf{Step 1.} We first prove (\ref{15}).
\par
Set
\begin{equation*}
\begin{array}{l}
\begin{array}{llll}
v(t)=(u-\varepsilon u_{xx})(t).
\end{array}
\end{array}
\end{equation*}
It follows from It\^{o}'s rule that
\begin{equation*}
\begin{array}{l}
\begin{array}{llll}
dv^{2}=2vdv+(dv)^{2}
\\~~~~~=2(u-\varepsilon u_{xx})[( u_{xx}-u^{3}+u)dt+g(u)dB]+g^{2}(u)dt
\\~~~~~=(2u u_{xx}-2u^{4}+2u^{2}-2\varepsilon| u_{xx}|^{2}+2\varepsilon u_{xx} \cdot u^{3} -2\varepsilon u_{xx} \cdot u)dt+2vg(u)dB+g^{2}(u)dt,
\end{array}
\end{array}
\end{equation*}
namely, we have
\begin{equation*}
\begin{array}{l}
\begin{array}{llll}
\|v(t)\|_{L^{2}(I)}^{2}+\int_{0}^{t}[2(1-\varepsilon)\| u_{x}\|_{L^{2}(I)}^{2}+2\int_{I}u^{4}dx+2\varepsilon\| u_{xx}\|_{L^{2}(I)}^{2}]ds
\\=\|v(0)\|_{L^{2}(I)}^{2}+2\int_{0}^{t}\|u\|_{L^{2}(I)}^{2}ds+2\varepsilon\int_{0}^{t}\int_{I}u_{xx} \cdot u^{3}dx ds+2\int_{0}^{t}(v,g(u))dB+\int_{0}^{t}\|g(u)\|_{L^{2}(I)}^{2}dt
\\=\|v(0)\|_{L^{2}(I)}^{2}+2\int_{0}^{t}\|u\|_{L^{2}(I)}^{2}ds-6\varepsilon\int_{0}^{t}\int_{I}| u_{x}|^{2}u^{2}dxds+2\int_{0}^{t}(v,g(u))dB+\int_{0}^{t}\|g(u)\|_{L^{2}(I)}^{2}ds
\\\leq\|v(0)\|_{L^{2}(I)}^{2}+2\int_{0}^{t}\|u\|_{L^{2}(I)}^{2}ds+2\int_{0}^{t}(v,g(u))dB+\int_{0}^{t}\|g(u)\|_{L^{2}(I)}^{2}ds.
\end{array}
\end{array}
\end{equation*}
\par
After some calculation, we obtain
\begin{equation*}
\begin{array}{l}
\begin{array}{llll}
(\sup\limits_{0\leq t\leq \tau}\|v(t)\|_{L^{2}(I)}^{2}+\int_{0}^{\tau}[2(1-\varepsilon)\|u_{x}\|_{L^{2}(I)}^{2}+2\int_{I}u^{4}dx+2\varepsilon\|u_{xx}\|_{L^{2}(I)}^{2}]dt)^{p}
\\\leq C(p)[\|v(0)\|_{L^{2}(I)}^{2p}+(\int_{0}^{\tau}\|u\|_{L^{2}(I)}^{2}dt)^{p}+\sup\limits_{0\leq t\leq \tau}|\int_{0}^{t}(v,g(u))dB|^{p}+(\int_{0}^{\tau}\|g(u)\|_{L^{2}(I)}^{2}dt)^{p}],
\end{array}
\end{array}
\end{equation*}
by the Burkholder-Davis-Gundy inequality, we have
\begin{equation*}
\begin{array}{l}
\begin{array}{llll}
\mathbb{E} \sup\limits_{0\leq t\leq \tau}\|v(t)\|_{L^{2}(G)}^{2p}+\mathbb{E}(\int_{0}^{\tau}\| u_{x}\|_{L^{2}(I)}^{2}dt)^{p}+\mathbb{E}(\int_{0}^{\tau}\int_{I}u^{4}dxdt)^{p}+\mathbb{E}(\int_{0}^{\tau}\varepsilon\|u_{xx}\|_{L^{2}(I)}^{2}dt)^{p}
\\\leq C(p)[\mathbb{E}\|v(0)\|_{L^{2}(I)}^{2p}+\mathbb{E}(\int_{0}^{\tau}\|u\|_{L^{2}(I)}^{2}dt)^{p}+\mathbb{E}\sup\limits_{0\leq t\leq \tau}|\int_{0}^{t}(v,g(u))dB|^{p}+\mathbb{E}(\int_{0}^{\tau}\|g(u)\|_{L^{2}(I)}^{2}dt)^{p}]
\\\leq C(p)[\mathbb{E}\|v(0)\|_{L^{2}(I)}^{2p}+\mathbb{E}(\int_{0}^{\tau}\|u\|_{L^{2}(I)}^{2}dt)^{p}+\rho\mathbb{E} \sup\limits_{0\leq t\leq \tau}\|v(t)\|_{L^{2}(I)}^{2p}+C(\rho)\mathbb{E}(\int_{0}^{\tau}\|g(u)\|_{L^{2}(I)}^{2}dt)^{p}]
\\\leq C(p)[\mathbb{E}\|v(0)\|_{L^{2}(I)}^{2p}+\mathbb{E}(\int_{0}^{\tau}\|u\|_{L^{2}(I)}^{2}dt)^{p}+\rho\mathbb{E} \sup\limits_{0\leq t\leq \tau}\|v(t)\|_{L^{2}(I)}^{2p}+C(\rho,L)\mathbb{E}(\int_{0}^{\tau}(1+\|u\|_{L^{2}(I)}^{2})dt)^{p}]

\\\leq C(p,\rho,L,T)[1+\mathbb{E}\|v(0)\|_{L^{2}(I)}^{2p}+\mathbb{E}(\int_{0}^{\tau}\|u\|_{L^{2}(I)}^{2}dt)^{p}]+\rho C(p)\mathbb{E} \sup\limits_{0\leq t\leq \tau}\|v(t)\|_{L^{2}(I)}^{2p}

\\\leq C(p,\rho,L,T)[1+\mathbb{E}\|v(0)\|_{L^{2}(I)}^{2p}+\sigma\mathbb{E}(\int_{0}^{\tau}\int_{I}u^{4}dxdt)^{p}+C(\sigma,T)]+\rho C(p)\mathbb{E} \sup\limits_{0\leq t\leq \tau}\|v(t)\|_{L^{2}(I)}^{2p}
.
\end{array}
\end{array}
\end{equation*}
By taking $\sigma<<1,\rho<<1,$ we have
\begin{equation*}
\begin{array}{l}
\begin{array}{llll}
\mathbb{E} \sup\limits_{0\leq t\leq \tau}\|v(t)\|_{L^{2}(I)}^{2p}+\mathbb{E}(\int_{0}^{\tau}\| u_{x}\|_{L^{2}(I)}^{2}dt)^{p}+\mathbb{E}(\int_{0}^{\tau}\int_{I}u^{4}dxdt)^{p}+\mathbb{E}(\int_{0}^{\tau}\varepsilon\|u_{xx}\|_{L^{2}(I)}^{2}dt)^{p}

\\\leq C(p,\rho,L,\sigma,T)[\mathbb{E}\|v(0)\|_{L^{2}(I)}^{2p}+1]
\\\leq C(p,L,T,I,u_{0}).
\end{array}
\end{array}
\end{equation*}
By the regularity theory of elliptic equation
\begin{equation*}
\begin{array}{l}
\left\{
\begin{array}{llll}
u-\varepsilon u_{xx}=v
\\u(0,t)=0=u(1,t),

\end{array}
\right.
\end{array}
\begin{array}{lll}
{\rm{in}}~I
\\
\\
\end{array}
\end{equation*}
we have
\begin{equation*}
\begin{array}{l}
\begin{array}{llll}
\|u(t)\|_{L^{2}(I)}\leq\|v(t)\|_{L^{2}(I)},
\end{array}
\end{array}
\end{equation*}
This implies that (\ref{15}) holds.
\par
\textbf{Step 2.} We shall prove (\ref{16}).
\par
According to Gagliardo-Nirenberg inequality, we have
\begin{equation*}
\begin{array}{l}
\begin{array}{llll}
\|u\|_{L^{6}(I)}\leq C\|u\|_{H^{1}(I)}^{\frac{1}{3}}\|u\|_{L^{2}(I)}^{\frac{2}{3}}
,
\end{array}
\end{array}
\end{equation*}
thus,
\begin{equation}\label{12}
\begin{array}{l}
\begin{array}{llll}
\mathbb{E}\int_{0}^{\tau}\|u^{3}\|_{L^{2}(I)}^{2}dt
\\=\mathbb{E}\int_{0}^{\tau}\|u\|_{L^{6}(I)}^{6}dt
\\\leq C\mathbb{E}\int_{0}^{\tau}\|u\|_{H^{1}(I)}^{2}\|u\|_{L^{2}(I)}^{4}dt
\\\leq C\mathbb{E}[(\int_{0}^{\tau}\|u\|_{H^{1}(I)}^{2}dt)\cdot\sup\limits_{0\leq t\leq \tau}\|u\|_{L^{2}(I)}^{4}]
\\\leq C\mathbb{E}[(\int_{0}^{\tau}\|u_{x}\|_{L^{2}(I)}^{2}dt)\cdot\sup\limits_{0\leq t\leq \tau}\|u\|_{L^{2}(I)}^{4}]
\\\leq C[\mathbb{E}(\int_{0}^{\tau}\|u_{x}\|_{L^{2}(I)}^{2}dt)^{2}+\mathbb{E}\sup\limits_{0\leq t\leq \tau}\|u\|_{L^{2}(I)}^{8}]
.
\end{array}
\end{array}
\end{equation}
In view of (\ref{15}) and (\ref{12}), there holds that $u^{3}-u\in L^{2}(\Omega;L^{2}(0,T;L^{2}(I))),$ moreover, $g(u)\in L^{2}(\Omega;L^{2}(0,T;H^{1}(I))),$ according to Proposition \ref{Pro1} (3), we have
\begin{equation*}
\begin{array}{l}
\begin{array}{llll}
\mathbb{E} \sup\limits_{0\leq t\leq \tau}(\| u_{x}(t)\|_{L^{2}(I)}^{2}+\varepsilon\|u_{xx}(t)\|_{L^{2}(I)}^{2})+\mathbb{E} \int_{0}^{\tau}\|u_{xx}(t)\|_{L^{2}(I)}^{2}dt
\\\leq C[\mathbb{E}(\|u_{0x}\|_{L^{2}(I)}^{2}+\|u_{0 xx}\|_{L^{2}(I)}^{2})+\mathbb{E} \int_{0}^{\tau}\|(u^{3}-u)(t)\|_{L^{2}(I)}^{2}dt+\mathbb{E} \int_{0}^{\tau}\|g(u)\|_{H^{1}(I)}^{2}dt].
\end{array}
\end{array}
\end{equation*}
With the help of (\ref{15}) and (\ref{12}), one finds that
\begin{equation*}
\begin{array}{l}
\begin{array}{llll}
\mathbb{E} \sup\limits_{0\leq t\leq \tau}(\| u_{x}(t)\|_{L^{2}(I)}^{2}+\varepsilon\|u_{xx}(t)\|_{L^{2}(I)}^{2})+\mathbb{E} \int_{0}^{\tau}\|u_{xx}(t)\|_{L^{2}(I)}^{2}dt
\\\leq C[\|u_{0}\|_{H^{2}(I)}^{2}+\mathbb{E}(\int_{0}^{\tau}\|u_{x}\|_{L^{2}(I)}^{2}dt)^{2}+\mathbb{E}\sup\limits_{0\leq t\leq \tau}\|u\|_{L^{2}(I)}^{8}+\mathbb{E} \int_{0}^{\tau}\|u\|_{H^{1}(I)}^{2}dt+C(T)]

\\\leq C (u_{0},T,I).
\end{array}
\end{array}
\end{equation*}
Namely, we prove (\ref{16}).
\par
\textbf{Step 3.} We shall prove $\mathbb{P}(\{\omega\in \Omega~|~\tau^{*}(\omega)=+\infty\})=1$.
\par
Indeed, by the Chebyshev inequality, (\ref{16}) and the definition of $u,$ we have
\begin{equation*}
\begin{array}{l}
\begin{array}{llll}
\mathbb{P}(\{\omega\in \Omega| \tau^{*}(\omega)<+\infty\})
\\=\lim\limits_{T\rightarrow +\infty}\mathbb{P}(\{\omega\in \Omega| \tau^{*}(\omega)\leq T\})
\\=\lim\limits_{T\rightarrow +\infty}\mathbb{P}(\{\omega\in \Omega|  \tau(\omega)=\tau^{*}(\omega)\})
\\=\lim\limits_{T\rightarrow +\infty}\lim\limits_{R\rightarrow +\infty}\mathbb{P}(\{\omega\in \Omega|  \tau_{R}(\omega)\leq\tau(\omega)\})
\\=\lim\limits_{T\rightarrow +\infty}\lim\limits_{R\rightarrow +\infty}\mathbb{P}(\{\omega\in \Omega| \sup\limits_{0\leq t\leq \tau}\| u(t)\|_{H^{2}(I)}^{2}\geq \sup\limits_{0\leq t\leq \tau_{R}}\| u(t)\|_{H^{2}(I)}^{2}\})
\\=\lim\limits_{T\rightarrow +\infty}\lim\limits_{R\rightarrow +\infty}\mathbb{P}(\{\omega\in \Omega| \sup\limits_{0\leq t\leq \tau}\|u(t)\|_{H^{2}(I)}^{2}\geq R^{2}\})
\\\leq\lim\limits_{T\rightarrow +\infty}\lim\limits_{R\rightarrow +\infty}\frac{\mathbb{E}\sup\limits_{0\leq t\leq \tau}\|u(t)\|_{H^{2}(I)}^{2}}{R^{2}}=0,
\end{array}
\end{array}
\end{equation*}
this show that
\begin{equation*}
\begin{array}{l}
\begin{array}{llll}
\mathbb{P}(\{\omega\in \Omega| \tau^{*}(\omega)=+\infty\})=1,
\end{array}
\end{array}
\end{equation*}
namely, $\tau_{\infty}=+\infty$ P-a.s.
\section{Proof of Theorem \ref{Th2}}
\subsection{A priori estimate of $\{u^{\varepsilon}\}_{0<\varepsilon<\frac{1}{2}}$}
In this section, we will establish the following estimate
\begin{equation}\label{19}
\mathbb{E}\sup\limits_{0\leq |\theta|\leq \delta} \int_{0}^{T}\|u^{\varepsilon}(t+\theta)-u^{\varepsilon}(t)\|_{L^{2}(I)}^{2}dt\leq C\delta.
\end{equation}
\par
Establishing this estimate directly for $u^{\varepsilon}$ is very difficulty, movetived by Section 2, we should establish estimate for $v^{\varepsilon},$
then by applying the regularity theory of elliptic equation, we can obtain the estimate for $u^{\varepsilon}.$
\par
It is easy to see that
\begin{equation*}
\begin{array}{l}
\begin{array}{llll}
 v^{\varepsilon}(t+\theta)-v^{\varepsilon}(t)=\int_{t}^{t+\theta}u^{\varepsilon}_{xx}(s)ds-\int_{t}^{t+\theta}(u^{\varepsilon3}-u^{\varepsilon})(s)ds+\int_{t}^{t+\theta}g(u^{\varepsilon}(s))dB,
\end{array}
\end{array}
\end{equation*}
which implies
\begin{equation}\label{11}
\begin{array}{l}
\begin{array}{llll}
~~~\|v^{\varepsilon}(t+\theta)-v^{\varepsilon}(t)\|_{L^{2}(I)}
\\\leq\|\int_{t}^{t+\theta}u^{\varepsilon}_{xx}(s)ds\|_{L^{2}(I)}
+\|\int_{t}^{t+\theta}(u^{\varepsilon3}-u^{\varepsilon})(s)ds\|_{L^{2}(I)}+\|\int_{t}^{t+\theta}g(u^{\varepsilon}(s))dB\|_{L^{2}(I)}
\\\leq\int_{t}^{t+\theta}\|u^{\varepsilon}_{xx}(s)\|_{L^{2}(I)}ds
+\int_{t}^{t+\theta}\|(u^{\varepsilon3}-u^{\varepsilon})(s)\|_{L^{2}(I)}ds+\|\int_{t}^{t+\theta}g(u^{\varepsilon}(s))dB\|_{L^{2}(I)}.
\end{array}
\end{array}
\end{equation}
Taking the square in both side of (\ref{11}), we have
\begin{equation*}
\begin{array}{l}
\begin{array}{llll}
\|v^{\varepsilon}(t+\theta)-v^{\varepsilon}(t)\|_{L^{2}(I)}^{2}
\\\leq(\int_{t}^{t+\theta}\|u^{\varepsilon}_{xx}(s)\|_{L^{2}(I)}ds
+\int_{t}^{t+\theta}\|(u^{\varepsilon3}-u^{\varepsilon})(s)\|_{L^{2}(I)}ds+\|\int_{t}^{t+\theta}g(u^{\varepsilon}(s))dB\|_{L^{2}(I)})^{2}

\\\leq C\theta\int_{t}^{t+\theta}(\|u^{\varepsilon}_{xx}\|_{L^{2}(I)}^{2}+\|u^{\varepsilon3}-u^{\varepsilon}\|_{L^{2}(I)}^{2})ds+C\|\int_{t}^{t+\theta}g(u^{\varepsilon}(s))dB\|_{L^{2}(I)}^{2}
\end{array}
\end{array}
\end{equation*}
\par
We can infer from (\ref{16}) and (\ref{12}) that
\begin{equation}\label{17}
\begin{array}{l}
\begin{array}{llll}
\mathbb{E}\int_{0}^{T}\int_{t}^{t+\delta}\|u^{\varepsilon}_{xx}\|_{L^{2}(I)}^{2}dsdt
\\\leq\delta \mathbb{E}\int_{0}^{T}\| u^{\varepsilon}_{xx}(t)\|_{L^{2}(I)}^{2}dt
\\\leq C\delta,
\\
\mathbb{E}\int_{0}^{T}\int_{t}^{t+\delta}\|u^{\varepsilon3}-u^{\varepsilon}\|_{L^{2}(I)}^{2}dsdt
\\=\delta \mathbb{E}\int_{0}^{T}\|u^{\varepsilon3}-u^{\varepsilon}\|_{L^{2}(I)}^{2}dt
\\\leq C\delta [\mathbb{E}(\int_{0}^{T}\|u^{\varepsilon}_{x}\|_{L^{2}(I)}^{2}dt)^{2}+\mathbb{E}\sup\limits_{0\leq t\leq T}\|u^{\varepsilon}\|_{L^{2}(I)}^{8}+\mathbb{E}\int_{0}^{T}\|u^{\varepsilon}\|_{L^{2}(I)}^{2}dt]
\\\leq C\delta.
\end{array}
\end{array}
\end{equation}

By the Burkholder-Davis-Gundy inequality and Young's inequality, we have
\begin{equation}\label{18}
\begin{array}{l}
\begin{array}{llll}
\mathbb{E} \sup\limits_{0\leq |\theta|\leq \delta} \int_{0}^{T}\|\int_{t}^{t+\theta}g(u^{\varepsilon}(s))dB\|_{L^{2}(I)}^{2}dt
\\\leq \mathbb{E}\int_{0}^{T}\sup\limits_{0\leq |\theta|\leq \delta}\|\int_{t}^{t+\theta}g(u^{\varepsilon}(s))dB\|_{L^{2}(I)}^{2}dt
\\=\int_{0}^{T}\mathbb{E}\sup\limits_{0\leq |\theta|\leq \delta}\|\int_{t}^{t+\theta}g(u^{\varepsilon}(s))dB\|_{L^{2}(I)}^{2}dt
\\\leq C \int_{0}^{T}\mathbb{E}\int_{t}^{t+\delta}\|g(u^{\varepsilon}(s))\|_{L^{2}(I)}^{2}dsdt
\\\leq C\delta \mathbb{E}\int_{0}^{T}\|g(u^{\varepsilon}(s))\|_{L^{2}(I)}^{2}dt
\\\leq C\delta \mathbb{E}\int_{0}^{T}(1+\|u^{\varepsilon}\|_{L^{2}(I)}^{2})dt
\\\leq C\delta.
\end{array}
\end{array}
\end{equation}
It follows from (\ref{17})-(\ref{18}) that
\begin{equation*}
\mathbb{E}\sup\limits_{0\leq |\theta|\leq \delta} \int_{0}^{T}\|v^{\varepsilon}(t+\theta)-v^{\varepsilon}(t)\|_{L^{2}(I)}^{2}dt\leq C\delta.
\end{equation*}
By the regularity theory of elliptic equation
\begin{equation*}
\begin{array}{l}
\left\{
\begin{array}{llll}
u^{\varepsilon}-\varepsilon u^{\varepsilon}_{xx}=v^{\varepsilon}
\\u^{\varepsilon}(0,t)=0=u^{\varepsilon}(1,t),
\end{array}
\right.
\end{array}
\begin{array}{lll}
{\rm{in}}~I\\
\\
\end{array}
\end{equation*}
we have
\begin{equation*}
\begin{array}{l}
\begin{array}{llll}
\|u^{\varepsilon}(t)\|_{L^{2}(I)}\leq\|v^{\varepsilon}(t)\|_{L^{2}(I)},
\end{array}
\end{array}
\end{equation*}
thus, we have (\ref{19}).

\subsection{Tightness property of $\{u^{\varepsilon}\}_{0<\varepsilon<\frac{1}{2}}$ in $L^{2}(0,T;H^{1}(I))$}
\par
We may rewrite Lemma \ref{L7} in the following more convenient form.
\par
By the same way as in \cite[P919]{S2}, according to the priori estimates (\ref{15})(\ref{16}) and (\ref{19}), we obtain that
\begin{lemma}
For any $1\leq p<\infty$ and for any sequences $\mu_{m},\nu_{m}$ converging to $0$ such that the series $\sum\limits_{m=1}^{\infty}\frac{\mu_{m}^{\frac{1}{2}}}{\nu_{m}}$ converges, $\{u^{\varepsilon}\}_{0<\varepsilon<\frac{1}{2}}$ is bounded in $X_{p,\mu_{m},\nu_{m}}^{2}$ (the explicit definition of the space $X_{p,\mu_{m},\nu_{m}}^{2}$ can be found in Section 2) for any $m.$
\end{lemma}
\par
Set
$$S=L^{2}(0,T;H^{1}(I))$$
and $\mathcal{B}(S)$ the $\sigma-$algebra of the Borel sets of $S.$
\par
For any $0<\varepsilon<\frac{1}{2},$ let $\Phi_{\varepsilon}$ be the map
\begin{equation*}
\begin{array}{l}
\begin{array}{llll}
\Phi_{\varepsilon}:\Omega\rightarrow S
\\~~~~~\omega\rightarrow u^{\varepsilon}(\omega),
\end{array}
\end{array}
\end{equation*}
and $\Pi_{\varepsilon}$ be a probability measure on $(S,\mathcal{B}(S))$ defined by
$$\Pi_{\varepsilon}(A)=\mathbb{P}(\Phi^{-1}_{\varepsilon}(A)), A\in \mathcal{B}(S).$$
\begin{proposition}\label{Pro3}
The family of probability measures $\{\Pi_{\varepsilon}:0<\varepsilon<\frac{1}{2}\}$ is tight in $S.$
\end{proposition}
\begin{proof}
For any $\rho >0,$ we should find the compact subsets
\begin{equation*}
\begin{array}{l}
\begin{array}{llll}
Y_{\rho}^{1}\subset L^{2}(0,T;H^{1}(I)),
\end{array}
\end{array}
\end{equation*}
such that
\begin{equation}\label{20}
\mathbb{P}(\omega:u^{\varepsilon}(\omega,\cdot)\notin Y_{\rho}^{1})\leq\rho.
\end{equation}

\par
Indeed, let $Y_{\rho}^{2}$ be a ball of radius $M_{\rho}$ in $Y_{\mu_{m},\nu_{m}}^{2}$ (the explicit definition of the space $Y_{\mu_{m},\nu_{m}}^{2}$ can be found in Section 2), centered at zero and with sequences $\mu_{m},\nu_{m}$ independent of $\varepsilon,$ converging to $0$ and such that the series $\sum\limits_{m=1}^{\infty}\frac{\mu_{m}^{\frac{1}{2}}}{\nu_{m}}$ converges. From Corollary \ref{c1}, $Y_{\rho}^{2}$ is a compact subset of $L^{2}(0,T;H^{1}(I)),$ and
\begin{equation*}
\mathbb{P}(\omega:u^{\varepsilon}(\omega,\cdot)\notin Y_{\rho}^{2})\leq \mathbb{P}(\omega:\|u^{\varepsilon}\|_{Y_{\mu_{m},\nu_{m}}^{2}}>M_{\rho})
\leq \frac{1}{M_{\rho}}\mathbb{E}\|u^{\varepsilon}\|_{Y_{\mu_{m},\nu_{m}}^{2}}\leq \frac{C}{M_{\rho}},
\end{equation*}
choosing $M_{\rho}=C\rho^{-1},$ we get (\ref{20}).
\par
This proves that
$$\Pi_{\varepsilon}(Y_{\rho}^{2})\geq 1-\rho, $$
for any $0<\varepsilon<\frac{1}{2}.$
\end{proof}

\subsection{The convergence result}
The main idea in this part comes from \cite{C2,C3}.
\par
\par
The proof of Theorem \ref{Th2} is divided into several steps.
\par
\textbf{Step 1.} We prove that $u^{\varepsilon}$ converges in probability to some random variable $z\in L^{2}(0,T;H^{1}(I)).$
\par
As proved in Proposition \ref{Pro3}, the family $\mathcal{L}(u^{\varepsilon})$ is tight in $L^{2}(0,T;H^{1}(I))$.
Then, due to the Skorokhod theorem for any two sequences $\{\varepsilon_{n}\}_{n\in N}$ and $\{\varepsilon_{m}\}_{m\in N}$ converging
to zero, there exist subsequences $\{\varepsilon_{n(k)}\}_{k\in N}$ and $\{\varepsilon_{m(k)}\}_{k\in N}$ and a sequence of random
elements
$$\{\rho_{k}\}_{k\in N}:=\{(u_{1}^{k},u_{2}^{k},\hat{B}_{k})\}_{k\in N}$$
in $L^{2}(0,T;H^{1}(I))\times L^{2}(0,T;H^{1}(I))\times C([0,T];\mathbb{R})$, defined on some probability space $(\hat{\Omega},\hat{\mathcal{F}},\hat{\mathbb{P}})$,
such that
$$\mathcal{L}(\rho_{k})=\mathcal{L}(u^{\varepsilon_{n(k)}},u^{\varepsilon_{m(k)}},B),$$
namely,
$$\mathcal{L}(u_{1}^{k},u_{2}^{k},\hat{B}_{k})=\mathcal{L}(u^{\varepsilon_{n(k)}},u^{\varepsilon_{m(k)}},B),$$
for each $k\in N$, and
$\rho_{k}$ converges $\hat{\mathbb{P}}$-a.s. to some random element $\rho:=(u_{1},u_{2},\hat{B})\in L^{2}(0,T;H^{1}(I))\times L^{2}(0,T;H^{1}(I))\times C([0,T];\mathbb{R})$.
\par
We now prove $u_{1}=u_{2}.$
\par
Indeed, according to the fact that $u_{1}^{k}$ and $u_{2}^{k}$ solve (\ref{13}) with $B$ replaced by $\hat{B}_{k},$ namely, we have
\begin{equation}
\begin{array}{l}
\left\{
\begin{array}{llll}
d(u_{1}^{k}-\varepsilon_{n(k)} u_{1xx}^{k})+(- u_{1xx}^{k}+u_{1}^{k3}-u_{1}^{k})dt=g(u_{1}^{k})d\hat{B}_{k}
\\u_{1}^{k}(0,t)=0=u_{1}^{k}(1,t)
\\u_{1}^{k}(0)=u_{0}

\end{array}
\right.
\end{array}
\begin{array}{lll}
{\rm{in}}~I\times(0,T)\\
{\rm{in}}~(0,T)\\
{\rm{in}}~I
\end{array}
\end{equation}
and
\begin{equation}
\begin{array}{l}
\left\{
\begin{array}{llll}
d(u_{2}^{k}-\varepsilon_{m(k)} u_{2xx}^{k})+(-u_{2xx}^{k}+u_{2}^{k3}-u_{2}^{k})dt=g(u_{2}^{k})d\hat{B}_{k}
\\u_{2}^{k}(0,t)=0=u_{2}^{k}(1,t)
\\u_{2}^{k}(0)=u_{0}

\end{array}
\right.
\end{array}
\begin{array}{lll}
{\rm{in}}~I\times(0,T)\\
{\rm{in}}~(0,T)\\
{\rm{in}}~I,
\end{array}
\end{equation}
it holds that
\begin{equation*}
\begin{array}{l}
\begin{array}{llll}
(u_{1}^{k}(t),\varphi)+\varepsilon_{n(k)}( u_{1x}^{k}(t),\varphi_{x})
\\~~~~~~=(u_{0},\varphi)+\varepsilon_{n(k)}(u_{0x},\varphi_{x})+\int_{0}^{t}( u_{1x}^{k}(s),\varphi_{x})ds+\int_{0}^{t}(u_{1}^{k3}-u_{1}^{k},\varphi)ds+\int_{0}^{t}(g(u_{1}^{k}),\varphi)d\hat{B}_{k}(s),
\\\\
(u_{2}^{k}(t),\varphi)+\varepsilon_{m(k)}(u_{2x}^{k}(t),\varphi_{x})
\\~~~~~~=(u_{0},\varphi)+\varepsilon_{m(k)}( u_{0x},\varphi_{x})+\int_{0}^{t}( u_{2x}^{k}(s),\varphi_{x})ds+\int_{0}^{t}(u_{2}^{k3}-u_{2}^{k},\varphi)ds+\int_{0}^{t}(g(u_{2}^{k}),\varphi)d\hat{B}_{k}(s).
\end{array}
\end{array}
\end{equation*}
It follows from Vitali's convergence theorem that
\begin{equation*}
\begin{array}{l}
\begin{array}{llll}
\lim\limits_{k\rightarrow\infty}\mathbb{E}\|u_{1}^{k}-u_{1}\|_{L^{2}(0,T;H^{1}(I))}^{2}=0,
\end{array}
\end{array}
\end{equation*}
according to this equality, Theorem \ref{Th1}, \cite[P284]{B2}, \cite[P1126,Lemma 2.1]{D1} and \cite[P151,Lemma 3.1]{G1}, it is easy to see for any $\delta>0$ and any $\varphi\in H^{1}_{0}(I),$ we have
\begin{equation*}
\begin{array}{l}
\begin{array}{llll}
\lim\limits_{k\rightarrow\infty}\mathbb{P}(\|(u_{1}^{k}(t),\varphi)-(u_{1}(t),\varphi)\|_{L^{2}(0,T)}>\delta)=0,
\\
\lim\limits_{k\rightarrow\infty}\mathbb{P}(\|\int_{0}^{t}( u_{1x}^{k}(s),\varphi_{x})ds-\int_{0}^{t}( u_{1x}(s),\varphi_{x})ds\|_{L^{2}(0,T)}>\delta)=0,
\\
\lim\limits_{k\rightarrow\infty}\mathbb{P}(\|\int_{0}^{t}(u_{1}^{k3}-u_{1}^{k},\varphi)ds-\int_{0}^{t}(u_{1}^{3}-u_{1},\varphi)ds\|_{L^{2}(0,T)}>\delta)=0,
\\
\lim\limits_{k\rightarrow\infty}\mathbb{P}(\|\int_{0}^{t}(g(u_{1}^{k}),\varphi)d\hat{B}_{k}(s)-\int_{0}^{t}(g(u_{1}),\varphi)d\hat{B}(s)\|_{L^{2}(0,T)}>\delta)=0.
\end{array}
\end{array}
\end{equation*}
By the same way, we have
\begin{equation*}
\begin{array}{l}
\begin{array}{llll}
\lim\limits_{k\rightarrow\infty}\mathbb{P}(\|(u_{2}^{k}(t),\varphi)-(u_{2}(t),\varphi)\|_{L^{2}(0,T)}>\delta)=0,
\\
\lim\limits_{k\rightarrow\infty}\mathbb{P}(\|\int_{0}^{t}(u_{2x}^{k}(s), \varphi_{x})ds-\int_{0}^{t}( u_{2x}(s), \varphi_{x})ds\|_{L^{2}(0,T)}>\delta)=0,
\\
\lim\limits_{k\rightarrow\infty}\mathbb{P}(\|\int_{0}^{t}(u_{2}^{k3}-u_{1}^{k},\varphi)ds-\int_{0}^{t}(u_{2}^{3}-u_{1},\varphi)ds\|_{L^{2}(0,T)}>\delta)=0,
\\
\lim\limits_{k\rightarrow\infty}\mathbb{P}(\|\int_{0}^{t}(g(u_{2}^{k}),\varphi)d\hat{B}_{k}(s)-\int_{0}^{t}(g(u_{2}),\varphi)d\hat{B}(s)\|_{L^{2}(0,T)}>\delta)=0.
\end{array}
\end{array}
\end{equation*}
\par
By taking the limit in probability as $k$ goes to infinity, we have
\begin{equation*}
\begin{array}{l}
\begin{array}{llll}
(u_{1}(t),\varphi)
=(u_{0},\varphi)+\int_{0}^{t}(u_{1x}(s),\varphi_{x})ds+\int_{0}^{t}(u_{1}^{3}-u_{1},\varphi)ds+\int_{0}^{t}(g(u_{1}),\varphi)d\hat{B}(s),
\\\\
(u_{2}(t),\varphi)
=(u_{0},\varphi)+\int_{0}^{t}(u_{2x}(s),\varphi_{x})ds+\int_{0}^{t}(u_{2}^{3}-u_{2},\varphi)ds+\int_{0}^{t}(g(u_{2}),\varphi)d\hat{B}(s).
\end{array}
\end{array}
\end{equation*}
Then, $u_{1},u_{2}$ coincide with the unique solution of heat equation perturbed by
the noise $\hat{B},$ thus $u_{1}=u_{2}.$
\par
It follows from Lemma \ref{L1} that $u^{\varepsilon}$ converges in probability to some random variable $z\in L^{2}(0,T;H^{1}(I))$.

\par
\textbf{Step 2.} We prove that $z$ is the solution of (\ref{21}).
\par
It follows from
\begin{equation*}
\begin{array}{l}
\begin{array}{llll}
\lim\limits_{\varepsilon\rightarrow0}\mathbb{P}( \|u^{\varepsilon}-z\|_{L^{2}(0,T;H^{1}(I))}>\delta)=0

\end{array}
\end{array}
\end{equation*}
that
\begin{equation*}
\begin{array}{l}
\begin{array}{llll}
\lim\limits_{\varepsilon\rightarrow0}\mathbb{P}(\|(u^{\varepsilon}(t),\varphi)-(z(t),\varphi)\|_{L^{2}(0,T)}>\delta)=0,
\\
\lim\limits_{\varepsilon\rightarrow0}\mathbb{P}(\|\int_{0}^{t}(u^{\varepsilon}_{x}(s),\varphi_{x})ds-\int_{0}^{t}(z_{x}(s),\varphi_{x})ds\|_{L^{2}(0,T)}>\delta)=0,
\\
\lim\limits_{\varepsilon\rightarrow0}\mathbb{P}(\|\int_{0}^{t}(u^{\varepsilon3}-u^{\varepsilon},\varphi)ds-\int_{0}^{t}(z^{3}-z,\varphi)ds\|_{L^{2}(0,T)}>\delta)=0,
\\
\lim\limits_{\varepsilon\rightarrow0}\mathbb{P}(\|\int_{0}^{t}(g(u^{\varepsilon}),\varphi)dB(s)-\int_{0}^{t}(g(z),\varphi)dB(s)\|_{L^{2}(0,T)}>\delta)=0.
\end{array}
\end{array}
\end{equation*}
Noting that
\begin{equation*}
\begin{array}{l}
\begin{array}{llll}
\mathbb{E}\sup\limits_{0\leq t\leq T}|\varepsilon(u^{\varepsilon}_{x}(t),\varphi_{x})|^{2}
\\\leq \mathbb{E}\sup\limits_{0\leq t\leq T}\varepsilon^{2}\|u^{\varepsilon}_{x}(t)\|_{L^{2}(I)}^{2}\|\varphi_{x}\|_{L^{2}(I)}^{2}
\\\leq \varepsilon\|\varphi_{x}\|_{L^{2}(I)}^{2}\mathbb{E}\sup\limits_{0\leq t\leq T}\varepsilon\|u^{\varepsilon}_{x}(t)\|_{L^{2}(I)}^{2},
\end{array}
\end{array}
\end{equation*}
we have
\begin{equation*}
\begin{array}{l}
\begin{array}{llll}
\lim\limits_{\varepsilon\rightarrow0}\mathbb{E}\sup\limits_{0\leq t\leq T}|\varepsilon(u^{\varepsilon}_{x}(t),\varphi_{x})|^{2}=0.
\end{array}
\end{array}
\end{equation*}
\par
By taking the limit in probability as $\varepsilon$ goes to zero in
\begin{equation*}
\begin{array}{l}
\begin{array}{llll}
(u^{\varepsilon}(t),\varphi)+\varepsilon(u^{\varepsilon}_{x}(t),\varphi_{x})
\\~~~~~~=(u_{0},\varphi)+\varepsilon(u_{0x},\varphi_{x})+\int_{0}^{t}( u^{\varepsilon}_{x}(s),\varphi_{x})ds+\int_{0}^{t}(u^{\varepsilon3}-u^{\varepsilon},\varphi)ds+\int_{0}^{t}(g(u^{\varepsilon}),\varphi)dB(s),

\end{array}
\end{array}
\end{equation*}
we deduce that
$z$ verifies the following equation $dt\otimes d\mathbb{P}-$almost everywhere:
\begin{equation*}
\begin{array}{l}
\begin{array}{llll}
(z(t),\varphi)
=(u_{0},\varphi)+\int_{0}^{t}( z_{x}(s),\varphi_{x})ds+\int_{0}^{t}(z^{3}-z,\varphi)ds+\int_{0}^{t}(g(z),\varphi)dB(s),

\end{array}
\end{array}
\end{equation*}
that is $z$ is the solution of (\ref{21}).

\noindent \footnotesize {\bf Acknowledgements.} \par I sincerely
 thank Professor Yong Li for many useful suggestions and help.\par
 \par

\end{document}